\documentclass[11pt]{scrartcl}

\usepackage[a4paper]{geometry}

\usepackage{scrhack}

\usepackage[]{amsmath}
\usepackage{mathtools}

\usepackage[utf8]{inputenc}
\usepackage[T1]{fontenc}
\usepackage{lmodern}
\usepackage{amssymb}
\usepackage{dsfont}

\usepackage[english]{babel}
\usepackage{microtype}

\usepackage[x11names]{xcolor}
\definecolor{googleblue}{HTML}{2A5DB0}
\definecolor{iconred}{HTML}{da4453}

\usepackage{float}
\usepackage{graphicx}

\usepackage{enumitem}

\usepackage{hyperref}
\hypersetup{
    colorlinks,
    linkcolor = {googleblue},
    citecolor = {iconred},
    urlcolor = {googleblue},
    bookmarksopen = false
}
\usepackage[noabbrev]{cleveref}

\usepackage{caption}
\setcapindent{0pt}
\addtokomafont{caption}{\footnotesize}
\addtokomafont{captionlabel}{\bfseries}

\usepackage[backend=biber,sorting=nyt,style=numeric, giveninits=true, maxbibnames=99,
	url=false,doi=false,isbn=false]{biblatex}
\usepackage{csquotes}

\DeclareFieldFormat{eprint:arxiv}{%
	\ifhyperref
		{\href{https://arxiv.org/abs/#1}{arXiv\addcolon\nolinkurl{#1}}}{\nolinkurl{#1}}%
	}
\DeclareFieldFormat{eprint:mrnumber}{%
	\ifhyperref
		{\href{http://www.ams.org/mathscinet-getitem?mr=MR#1}{MR#1}}{MR#1}%
	}
\DeclareFieldFormat{eprint:doi}{%
	\ifhyperref
		{\href{https://www.doi.org/#1}{doi\addcolon\nolinkurl{#1}}}{\nolinkurl{#1}}%
	}

\DeclareSourcemap{
	\maps[datatype=bibtex]{
		\map{
			\step[fieldsource=mrnumber, fieldtarget=eprint, final]
			\step[fieldset=eprinttype, fieldvalue=mrnumber]
		}
		\map{
			\step[fieldsource=arxiv, fieldtarget=eprint, final]
			\step[fieldset=eprinttype, fieldvalue=arxiv]
		}
	}
}

\renewbibmacro*{eprint}{%
	\printfield{arxiv}%
	\newunit\newblock
	\printfield{mrnumber}%
	\newunit\newblock
	\iffieldundef{eprinttype}
		{\printfield{eprint}}
		{\printfield[eprint:\strfield{eprinttype}]{eprint}}}

\usepackage{amsthm}

\usepackage[normalem]{ulem}

\theoremstyle{plain}
\newtheorem{theorem}{Theorem}[section]
\newtheorem{proposition}[theorem]{Proposition}
\newtheorem{lemma}[theorem]{Lemma}

\theoremstyle{definition}
\newtheorem{definition}[theorem]{Definition}
\newtheorem{remark}[theorem]{Remark}

\theoremstyle{remark}

\makeatletter
\DeclareFontFamily{U}{MnSymbolA}{}
\DeclareFontShape{U}{MnSymbolA}{m}{n}{
	<-6>  MnSymbolA5
	<6-7>  MnSymbolA6
	<7-8>  MnSymbolA7
	<8-9>  MnSymbolA8
	<9-10> MnSymbolA9
	<10-12> MnSymbolA10
	<12->   MnSymbolA12}{}
\DeclareFontShape{U}{MnSymbolA}{b}{n}{
	<-6>  MnSymbolA-Bold5
	<6-7>  MnSymbolA-Bold6
	<7-8>  MnSymbolA-Bold7
	<8-9>  MnSymbolA-Bold8
	<9-10> MnSymbolA-Bold9
	<10-12> MnSymbolA-Bold10
	<12->   MnSymbolA-Bold12}{}
\DeclareSymbolFont{MnSyA}{U}{MnSymbolA}{m}{n}
\SetSymbolFont{MnSyA}{bold}{U}{MnSymbolA}{b}{n}
\DeclareMathSymbol{\leftharpoondown}{\mathrel}{MnSyA}{'112}
\DeclareMathSymbol{\rightharpoonup}{\mathrel}{MnSyA}{'100}
\DeclareMathSymbol{\mn@relbar}{\mathrel}{MnSyA}{'320}
\def\leftharpoonfill@{\arrowfill@\leftharpoondown\mn@relbar\mn@relbar}
\def\rightharpoonfill@{\arrowfill@\mn@relbar\mn@relbar\rightharpoonup}
\DeclareRobustCommand{\overleftharpoon}{\mathpalette{\overarrow@\leftharpoonfill@}}
\DeclareRobustCommand{\overrightharpoon}{\mathpalette{\overarrow@\rightharpoonfill@}}
\makeatother
\newcommand*{\edge}[1]{\overrightharpoon{#1}}
\newcommand*{\textedge}[1]{\smash[b]{\overrightharpoon{#1}}}

\newcommand*{\comp}[1]{#1^{\!\mathrm{c}}}
\newcommand*{\boundary}[1]{\partial #1}
\newcommand*{\euler}{e}
\newcommand*{\ind}{\mathds{1}}

\title{Dynamical Gibbs-non-Gibbs transitions in \\ Widom-Rowlinson models on trees}

\author{Sebastian Bergmann\footnote{Ruhr-Universit\"at Bochum, Fakult\"at f\"ur Mathematik, Universit\"atsstra\ss e 150, 44780 Bochum, Germany. E-mail: sebastian.bergmann@rub.de, sascha.kissel@rub.de, christof.kuelske@rub.de} \and Sascha Kissel\footnotemark[\value{footnote}] \and Christof K\"ulske\footnotemark[\value{footnote}]}
\date{\today}

\begin{document}
\maketitle
\begin{abstract}
{\noindent\textbf{Abstract:} We consider the soft-core Widom-Rowlinson model for particles with spins and holes, on a Cayley tree of order $d$ (which has $d+1$ nearest neighbours), depending on repulsion strength $\beta$ between particles of different signs and on an activity parameter $\lambda$ for particles. We analyse Gibbsian properties of the time-evolved intermediate Gibbs measure of the static model, under a spin-flip time evolution, in a regime of large repulsion strength $\beta$.

We first show that there is a dynamical transition, in which the measure becomes non-Gibbsian at large times, independently of the particle activity, for any $d \geq 2$. In our second and main result, we also show that for large $\beta$ and at large times, the measure of the set of bad configurations (discontinuity points) changes from zero to one as the particle activity $\lambda$ increases, assuming that $d \geq 4$. Our proof relies on a general zero-one law for bad configurations on the tree, and the introduction of a set of uniformly bad configurations given in terms of subtree percolation, which we show to become typical at high particle activity.}
\end{abstract}
\vfill
\begin{center}
\end{center}
\clearpage

\section{Introduction}

Dynamical Gibbs-non-Gibbs transitions for spin models can be analysed on various types of graphs. A prototypical example is the low-temperature Ising model on the integer lattice in two or more dimensions, under stochastic independent spin-flip dynamics \cite{van_enter_possible_2002}. The authors showed in particular that the time-evolved model, started from the initial plus Gibbs measure in zero external field, fails to be Gibbsian for all large enough finite times, while the Gibbs property is preserved for small times.

This study of the Ising model has been extended to mean-field and Kac-models, where the appropriate notion of sequential Gibbsianness has been employed \cite{fernandez_variational_2013,fernandez_variational_2014,den_hollander_gibbs-non-gibbs_2015,kulske_spin-flip_2007}. For studies of Gibbsian properties of the Potts model under time evolution and also other transformations, see \cite{haggstrom_is_2003,haggstrom_gibbs_2004,henning_gibbsnon-gibbs_2019,kulske_dynamical_2021}.

A famous model in the world of point particles taking positions in $\mathbb{R}^d$ is the Euclidean Widom-Rowlinson model. In its original form the particles carry one of the possible signs plus or minus and are subjected to hardcore pair interactions, which forbids inter-particle distances smaller than a fixed radius $R>0$ when they carry different signs. In equilibrium the continuum model shows a ferromagnetic transition, see \cite{bricmont_structure_1984,chayes_analysis_1995}, for dynamical Gibbs-non-Gibbs transitions, see \cite{jahnel_widom-rowlinson_2017}, for generalities on Gibbsian point processes see \cite{dereudre_existence_2012}.

In this note we turn to the soft-core Widom-Rowlinson model on trees under a spin-flip time evolution. The model has the local state space $\{-1,0,1\}$, where the spin value $0$ stands for an empty site. It has two parameters $\lambda>0$, the activity of particles, and $\beta>0$ describing a soft-core repulsion between neighbouring particles of different signs.

Static and dynamical soft-core Widom-Rowlinson models have been studied on the lattice and in mean field, see \cite{higuchi_results_2004,kissel_dynamical_2019,kissel_dynamical_2020,kozitsky_phase_2018}.

The static behaviour of the soft-core Widom-Rowlinson model on a Cayley tree of order $d \geq 2$ is known, see \cite{kissel_hard-core_2019}, where $d$-dependent non-uniqueness regions in the space of the parameters $\beta,\lambda$ are described. In these regions there exist at least three different tree-indexed Markov chain Gibbs measures (splitting Gibbs measures) which are tree-automorphism invariant. Among these there is a unique measure which is also spin-flip invariant, the so-called \emph{intermediate measure} $\smash[b]{\mu^{\#}_{\beta,\lambda}}$. In this note we focus on the corresponding \emph{time-evolved intermediate measure} $\smash[b]{\mu^{\#}_{\beta,\lambda,t}}$ which is obtained by drawing the initial condition w.r.t. the infinite-volume measure $\smash[b]{\mu^{\#}_{\beta,\lambda}}$ and applying independent stochastic spin-flips to the spins at the occupied sites while keeping the holes fixed.

Previous study has shown that in the time-evolved Ising model in zero external field a quite unusual behaviour (compared to lattices) occurs on regular trees \cite{van_enter_gibbs-non-gibbs_2012}. One feature was that at large $\beta$ and sufficiently large times \emph{all configurations} become \emph{bad} for the time-evolved intermediate measure $\mu^{\#}_{\mathrm{Ising};\beta,t}$. Bad configurations are non-removable discontinuity points of finite-volume conditional probabilities in the sense of Definition \ref{def:essential_discontinuity}. If there exists at least one bad point, a Gibbsian representation with a well-behaved specification is impossible. If almost all (or even all) configurations are bad, this signals a particularly strong internal non-locality of the system.

As for the Widom-Rowlinson model, the intermediate Ising measure can be defined to be the unique spin-flip invariant tree-automorphism invariant Gibbs measure which is also a tree-indexed Markov chain. The full measure badness was seen only in the time-evolved intermediate Ising measure, for different Ising measures as starting measure even recovery of the Gibbs property for large times was proved in \cite{van_enter_gibbs-non-gibbs_2012}.

On the other hand, full-measure badness for time-evolved Widom-Rowlinson measures was discovered in models with hard-core interaction, both in the continuum and on the lattice \cite{jahnel_widom-rowlinson_2017,kissel_dynamical_2020}. The mechanisms responsible for this in both models were based on the hard-core interactions and were much different from the mechanism responsible for the result for the Ising model on a tree. All studies of the time-evolved Widom-Rowlinson model with soft-core interactions on discrete graphs have shown non-Gibbsian behaviour at large times, but did not reveal full measure bad configurations.

How much of this all-badness can we expect to survive in the time-evolved intermediate measure, and what is the role of the second parameter, the activity $\lambda$?

\subsection*{Results and some proof ideas}

Our results split into criteria for badness of the time-evolved measure given in Subsection \ref{subsec:results_badness} and criteria for goodness of the time-evolved measure given in Subsection \ref{subsec:results_goodness}.

Let us explain informally our main badness result of Theorem \ref{thm:almost-sure-badness} which we consider to be the most interesting part. It states that on Cayley trees of order $d \geq 4$ for large enough repulsion strength $\beta$, and large activity $\lambda$, at large enough times, the time-evolved measure has bad configurations of full measure. Bad configurations are by definition discontinuity points of conditional probabilities of the time-evolved measure in the sense of Definition \ref{def:essential_discontinuity}.

To appreciate this result note that our starting measure is provably non-extremal in a regime of large $\beta$ and large $\lambda$ by the Kesten-Stigum criterion \cite{kesten_additional_1966}, see Proposition \ref{lemma:non-extremality_intermediate}. While bad configurations generally form a tail-event, our measure of interest will not be extremal and therefore not tail-trivial, therefore there is a priori no reason why their probabilities should be restricted to zero or one. However, we are able to formulate a different zero-one law in Theorem \ref{thm:zero-one_law} which holds in all parameter regimes, and does not assume tail-triviality, but from which the desired statement follows. It applies more generally not only to the intermediate measure, but to all homogeneous tree-indexed Markov chain Gibbs measures of our model. The proof we give uses a representation of configurations of the measure in terms of the (finite or infinite) connected components by means of a renewal construction on the tree, starting from an arbitrary root, see Subsection \ref{subsec:zero-one}.

Using our zero-one law of Theorem \ref{thm:zero-one_law}, we notice that for the proof of the full-measure badness in Theorem \ref{thm:almost-sure-badness}, it suffices to find positive measure sets of bad configurations. It turns out that sets with this property can be given in terms of the subtree percolation condition of Theorem \ref{thm:condition_bad_configurations}. This condition on infinite-volume configurations of the Widom-Rowlinson model looks only at the occupied sites, i.e. the sites with spin-values not equal to zero, disregarding the signs, and asks for the existence of an infinite occupied subtree of large enough order $s$.

A glance at the result of Theorem \ref{thm:condition_bad_configurations} then also shows that for any order $d\geq 2$ there is a dynamical transition, in which the measure becomes non-Gibbsian at large times, for large $\beta$, independently of the particle activity. This is clear, as it suffices to exhibit just one particular bad configuration, and for $s=d$ the theorem shows that the fully occupied configuration is always bad.

Let us outline some of the ideas and difficulties of the proof of Theorem \ref{thm:condition_bad_configurations}, i.e. explain how to ensure badness of configurations with percolating subtrees. The detailed proof will be given in Subsection \ref{subsec:subtree_badness}. Starting from a two-layer representation of the conditional probabilities of the time-evolved model one is first led to an inhomogeneous recursion on a subtree of occupied sites which needs to be run on the first layer (spin configurations at time zero). In this recursion the influence of the configuration in the conditioning on the second layer (infinite-volume configurations at time $t$) appears as an inhomogeneous magnetic-field term, and unoccupied sites act as a dilution on the first layer. Non-removable discontinuities (bad points of the time-evolved measure) come from non-decaying memory on the boundary condition in this recursion, and this is what needs to be proved to ensure badness. There are mainly two possibly counteracting influences in this recursion, and this is what creates new difficulty for the Widom-Rowlinson model as compared to the non-homogeneous recursion for the Ising model \cite{bissacot_stability_2017,van_enter_gibbs-non-gibbs_2012}. First, there are the terms coming from the infinite tree of occupied sites attached at the origin, from we have cut off all finite parts. As for the second part, there are also contributions caused by the non-percolating appendices of the tree. These may work in the opposite direction, depending on the choice of the signs of the configuration. The proof then consists in showing that, not only are the percolating parts able to carry a boundary condition to the origin regardless of the second-layer spins, but they also win against possible counteractions from the finite parts. The proof of the all-badness result of Theorem \ref{thm:almost-sure-badness} is then completed by ensuring typicality of $s$-subtree percolation of occupied sites, see Proposition \ref{prop:probability_subtrees}.

In Section \ref{subsec:results_goodness} we give two results on goodness. The first is the small-time Gibbs property of Theorem \ref{thm:almost-sure-goodness_dobrushin}, which follows by the Dobrushin method. The second result of Theorem \ref{thm:almost-sure-goodness} asserts that at any possibly large $\beta$, the set of bad configurations, while not necessarily empty, has zero measure, if the activity $\lambda$ is sufficiently small. The proof is carried out again in the two-layer picture, via comparison to extinction of a suitable Galton-Watson tree.

Finally, from Theorem \ref{thm:almost-sure-badness} and Theorem \ref{thm:almost-sure-goodness} we conclude, that for any fixed large enough $\beta$ and large enough time, a $\lambda$-driven transition between a non-Gibbs regime with zero measure bad configurations to a regime with full measure bad configurations occurs.

The remainder of the paper is organized as follows. In Subsection \ref{subsec:definition_notation} we define the soft-core Widom-Rowlinson model, review the relation between tree-indexed Markov chains and boundary laws provided by Zachary's Theorem \ref{thm:bl_gibbs_measure}, and define the time-evolved measure and notion of a bad configuration. Subsection \ref{subsec:results_badness} contains our results on badness, Subsection \ref{subsec:results_goodness} contains our results on goodness, and Sections \ref{sec:proofs_badness} and \ref{sec:proofs_goodness} contain the proofs.

\section{Model and main results}

\subsection{Definitions and notation}
\label{subsec:definition_notation}

\subsubsection{The soft-core Widom-Rowlinson model on trees}

A \emph{graph} $(V\!,E)$ is a \emph{vertex set} $V$ in combination with a set of \emph{edges} $E \subset V^2$. If two vertices are neighbours in the sense that they are connected through an edge we write $i \sim j$. We will denote the set of \emph{oriented edges}, which is the set of ordered pairs in $E$, as $\textedge{E}$ and its elements by $\textedge{ij}$ or in unambiguous cases $ij$ to lighten the notation. Given a subset $\Lambda \subset V$ of the vertex set we write $\Lambda\Subset V$ if it is finite and denote its \emph{boundary} by $\boundary{\Lambda}$, this is the set of vertices in $\comp{\Lambda}$ directly connected to $\Lambda$ through an edge:
\begin{equation}
\label{eq:boundary_definition}
\boundary{\Lambda}%
:= \big\{i\in \comp{\Lambda}\,\big|\,\exists j\in\Lambda,\,i\sim j\big\}.%
\end{equation}
We sometimes abuse this notation and write $\boundary{i}$ for the neighbours of $i \in V$. A \emph{path} between two vertices $i,j\in V$ will be a finite collection of vertices $(k_0,\ldots,k_N)$ such that $k_0 = i$, $k_N = j$ and $k_n \sim k_{n+1}$ for all $n = 0,\ldots,N-1$. We will call this path non-repeating (or equivalently self-avoiding) if $k_n \neq k_m$ for all $n \neq m$. If at least one such path exists between every two vertices of a subset $\Lambda$ we will call $\Lambda$ \emph{connected}. In particular we consider trees, connected graphs where each vertex has a finite number of neighbours and where only one unique non-repeating path between two vertices $i,j \in V$ exists. We denote this path by $\mathcal{P}(i,j)$, the length of the path defines a metric on the tree via $d(i,j) := N\,$ if $\ \mathcal{P}(i,j)=(k_0,\ldots,k_N)$. A tree is called Cayley tree of order $d$ if each vertex has exactly $d+1 \in \mathbb{N}$ neighbours.

For the Widom-Rowlinson model we introduce a copy of the spin space $\{-1,0,1\}$ on each vertex. Combined they form the \emph{configuration space} $\Omega:=\{-1,0,1\}^V$ which is endowed with the $\sigma$-algebra $\mathcal{F}:=\mathcal{P}(\{-1,0,1\})^{\otimes V}$. Let $\Lambda \subset V$ be any set, we denote by $\Omega_{\Lambda}$ the set of configurations $\omega_{\Lambda} := (\omega_i)_{i \in \Lambda}$ restricted on $\Lambda$, and by $\sigma_{\Lambda}:\Omega \rightarrow \Omega_{\Lambda}$ the mapping with $\sigma_{\Lambda}(\omega)=\omega_{\Lambda}$ for each $\omega \in \Omega$. We write $\omega_{A}\eta_{B} \in \Omega_{A \cup B}$ for the concatenation of configurations on disjoint sets $A, B \subset V$. The $\sigma$-algebra on $\Omega$ generated by the projections to a subset $\Lambda \subset V$ is denoted by $\mathcal{F}_{\Lambda}$. If a function $f\,:\,\Omega \rightarrow \mathbb{R}$ is $\mathcal{F}_\Lambda$-measurable for $\Lambda \Subset V$, $f$ is called a local function. A function $f$ is called quasilocal on $\Omega$ if there exists a sequence of local functions $(f_n)_{n\in \mathbb{N}}$ with $\lim_{n\rightarrow \infty} \Vert f-f_n\Vert_{\infty} = 0$. Note that for finite state spaces quasilocality is equivalent to continuity with respect to the product topology.

The soft-core Widom-Rowlinson model is a natural extension of the original Widom-Rowlinson model on graphs where neighbouring vertices are not allowed to take the same spin value. For our model this hard-core type restriction is weakened. Here, it is allowed that vertices with $+1$ and $-1$ spins are nearest neighbours, however this will be punished by a repulsion parameter $\beta>0$ called the inverse temperature. The interaction of the model can be written as a potential
\begin{equation}
\Phi_{\Lambda}(\omega) = \left\{%
\begin{array}{ccc}
\beta \ind_{\{\omega_i\omega_j =-1\}} && \text{if}\ \Lambda =\{i,j\} \in E\\%
-h\omega_i -\log(\lambda)\omega_i^2 && \text{if}\ \Lambda = \{i\}\\%
0 && \text{else}%
\end{array} \right.%
\end{equation}
where $h\in\mathbb{R}$ is an external magnetic field and $\lambda>0$ serves as an activity parameter of occupied sites. Throughout this paper we consider the case of $h=0$.

To define Gibbs measures we need the notion of specifications. These are a families of probability kernels $\gamma=(\gamma_\Lambda)_{\Lambda \Subset V}$ from $\mathcal{F}_{\comp{\Lambda}}$ to $\mathcal{F}$ respectively, which satisfy the properness condition $\gamma_\Lambda(A|\,\cdot\,) = \ind_A(\,\cdot\,)$ for all $A \in \mathcal{F}_{\comp{\Lambda}}$, and the consistency condition $\gamma_{\Delta}\gamma_{\Lambda} = \gamma_{\Delta}$ for all $\Lambda \subset \Delta \Subset V$. A specification is called quasilocal if for each $\Lambda \Subset V$ and each quasilocal function $f: \Omega\rightarrow \mathbb{R}$ the function
\begin{equation}
\gamma_\Lambda(f|\,\cdot\,) := \int_\Omega \gamma_\Lambda(d\omega|\,\cdot\,)f(\omega)%
\end{equation}
is quasilocal. We say a measure $\mu$ on $(\Omega,\mathcal{F})$ is a Gibbs measure, if it satisfies the Dobrushin-Lanford-Ruelle equations for a quasilocal specification, i.e.
\begin{equation}
\mu = \mu\gamma_{\Lambda}%
\end{equation} 
for each $\Lambda\Subset V$. This condition is equivalent to $\mu_{\Lambda}(f|\,\cdot\,) := \mu(f|\mathcal{F}_{\comp{\Lambda}})(\,\cdot\,) = \gamma_\Lambda(f|\,\cdot\,)$ $\mu$-almost-surely for each measurable function $f:\Omega \rightarrow \mathbb{R}$. Given a specification $\gamma$ we write $\mathcal{G}(\gamma)$ for the set of all Gibbs measures admitted by this specification. As $\mathcal{G}(\gamma)$ is a simplex we are particularly interested in its extremal points, the \emph{extremal Gibbs measures}.

For the potential of the soft-core Widom-Rowlinson model we can define a specification through the probability kernels defined by
\begin{equation}
\gamma_\Lambda(\sigma_{\Lambda}=\omega_\Lambda|\eta_{\comp{\Lambda}})%
= \frac{1}{Z_{\Lambda}(\eta_{\comp{\Lambda}})}%
\exp\big(-\mathcal{H}_\Lambda(\omega_{\Lambda}\eta_{\comp{\Lambda}})\big)%
\qquad \omega,\eta \in \Omega,%
\end{equation}
with the finite-volume Hamiltonian $\mathcal{H}_{\Lambda}(\omega) = \sum_{A\cap\Lambda\neq\emptyset,\,A\Subset V} \Phi_{A}(\omega)$ for all $\Lambda\Subset V$. The function $Z_{\Lambda}$ is called \emph{partition function} and is defined such that $\gamma_{\Lambda}(\,\cdot\,|\eta_{\comp{\Lambda}})$ is a probability measure for each $\eta$ in $\Omega$.

\subsubsection{Tree-indexed Markov chains and boundary laws}

We briefly review the notion of tree-indexed Markov chains and Zachary's theorem stating a one-to-one correspondence between boundary laws and Markov chain Gibbs measures. We begin with the definition of a Markov specification
\begin{definition}
\label{def:markov_specification}
A specification $\gamma=(\gamma_{\Lambda})_{\Lambda \Subset V}$ is called a \emph{Markov specification} if for each region $\Lambda\Subset V$ and all fixed spin configurations $\omega_{\Lambda}\in\Omega_{\Lambda}$ the specification density $\gamma_{\Lambda}(\omega_{\Lambda}|\,\cdot\,)$ is $\mathcal{F}_{\boundary{\Lambda}}$-measurable.
\end{definition}
One can easily see that the specification associated to the Widom-Rowlinson model is Markovian. To define the stronger notion of tree-indexed Markov chains we need a concept of past. We define the set of vertices that lies in the past of an oriented edge $\textedge{ij}$ as those vertices whose unique path to $i$ does not pass over the edge $\textedge{ij}$ and thereby does not contain $j$:
\begin{equation}
(-\infty,\edge{ij}) := \big\{k \in V\,\big|\, j\notin\mathcal{P}(i,k)\big\}.%
\end{equation}
This leads to the definition of a tree-indexed Markov chain.
\begin{definition}
\label{def:tree_indexed_markov_chain}
Let $(V\!,E)$ be a tree. A measure $\mu$ is called a tree-indexed Markov chain if
\begin{equation}
\label{eq:markov_chain_property}
\mu\big(\sigma_j = \omega_j\,\big|\,\mathcal{F}_{\!(-\infty,\edge{ij})}\big)%
= \mu\big(\sigma_j = \omega_j\,\big|\,\mathcal{F}_{i}\big)%
\qquad \mu-a.s. \quad \forall \omega_j \in \Omega_j,%
\end{equation}
for each edge $\textedge{ij} \in \textedge{E}$.
\end{definition}
There are some connections between Gibbs measures and tree-indexed Markov chains. To explain these relations we need to introduce boundary laws and transfer operators for Markov chains.
\begin{definition}
\label{def:transfer_operator}
We define the family of \emph{transfer operators} $(Q_{\{i,j\}})_{\{i,j\}\in E}$ of the Widom-Rowlinson model for each $\{i,j\} \in E$ by
\begin{equation}
\label{eq:transfer_operator}
Q_{\{i,j\}}(\omega_i,\omega_j)%
:=\exp\left(-\Phi_{\{i,j\}}(\omega_i,\omega_j)%
-\frac{\Phi_{\{i\}}(\omega_i)}{|\boundary{i}|}%
-\frac{\Phi_{\{j\}}(\omega_j)}{|\boundary{j}|}\right).%
\end{equation}
A family of vectors $\smash[b]{(l_{ij})_{ij\in \edge{E}}}$ with each $l_{ij}$ in $(0,\infty)^{\Omega_i}$ is called a \emph{boundary law} consistent with the transfer operators $(Q_{\{i,j\}})_{\{i,j\}\in E}$, if for each $\textedge{ij} \in \textedge{E}$ there exists a positive constant $c_{ij}>0$ such that the consistency equation
\begin{equation}
\label{eq:boundary_law_recursion}
l_{ij}(\omega_i) = c_{ij} \prod_{k \in \boundary{i}\setminus j}%
\sum_{\omega_k\in \Omega_k}Q_{\{k,i\}}(\omega_k,\omega_i)l_{ki}(\omega_k)%
\end{equation}
holds for every $\omega_i \in \Omega_i$.
\end{definition}
Note that boundary laws $\smash[b]{(l_{ij})_{ij\in \edge{E}}}$ are uniquely determined up to a constant only. Hence, one can choose one of the entries of the vectors $l_{ij}$ arbitrarily. In our case it is useful to set $l_{ij}(0)=1$ for every $\textedge{ij} \in \textedge{E}$. With the idea of transfer operators the specification of the Widom-Rowlinson can be rewritten as
\begin{equation}
\gamma_\Lambda(\sigma_{\Lambda}=\omega_\Lambda|\omega_{\Lambda^c} )%
= \frac{1}{Z'_\Lambda(\omega_{\comp{\Lambda}})} \prod_{\substack{\{i,j\}\in E\\ \{i,j\}\cap \Lambda \neq \emptyset}}%
Q_{\{i,j\}}(\omega_i,\omega_j)
\end{equation}
which leads to the following theorem, firstly proven by Zachary: 
\begin{theorem}[Zachary \cite{zachary_countable_1983}]
\label{thm:bl_gibbs_measure}
Let $\gamma=(\gamma_{\Lambda})_{\Lambda\Subset V}$ be a Markov specification on $(\Omega,\mathcal{F})$ with transfer operators $(Q_{\{i,j\}})_{\{i,j\}\in E}$. Then each boundary law $(l_{ij})_{ij\in\textedge{E}}$ consistent with these transfer operators defines a unique tree-indexed Markov chain $\mu \in \mathcal{G}(\gamma)$ via the equation
\begin{equation}
\label{eq:boundary_law_representation}
\mu(\sigma_{\Lambda\cup\boundary{\Lambda}}=\omega_{\Lambda\cup\boundary{\Lambda}})%
=\frac{1}{Z''_{\Lambda}}%
\prod_{\substack{\{i,j\} \in E \\ \{i,j\}\cap\Lambda\neq\emptyset}}%
Q_{\{i,j\}}(\omega_i,\omega_j)%
\prod_{k\in\boundary{\Lambda}}l_{kk_{\Lambda}}(\omega_k)%
\end{equation}
where $\Lambda\Subset V$ is a connected set and $Z''_{\Lambda}\in (0,\infty)$ a suitable normalizing constant. Here $k_{\Lambda}$ denotes the unique nearest neighbour vertex of $k \in \boundary{\Lambda}$ that lies inside of $\Lambda$.

Conversely, every tree-indexed Markov chain $\mu \in \mathcal{G}(\gamma)$ has the above representation through a boundary law $(l_{ij})_{ij \in \textedge{E}}$ which is uniquely defined up to a positive factor.
\end{theorem}
In \cite{kissel_hard-core_2019} the authors find solutions of the above recursion problem \eqref{eq:boundary_law_representation} for the soft-core model depending on the values of $\beta$ and $\lambda$. One of the results is that there always exists a solution with the property $l_{ij}(1) = l_{ij}(-1)$ which via Theorem \ref{thm:bl_gibbs_measure} defines a tree-indexed Markov chain and Gibbs measure $\mu^\#$ which we call intermediate measure. It is the only Gibbs measure if $\beta$ is small enough.

\subsubsection{Time evolution, good and bad configurations}

As Gibbs measures describe the equilibrium states of physical systems at constant temperatures, to investigate the behaviour of systems for variable temperatures we need to introduce a suitable transformation. In our case we use a Markovian semigroup $(\pi_t)_{t \in [0,\infty)}$ describing a site-independent spin-flip dynamics whose single-site marginals are given by
\begin{equation}
\label{eq:single_site_kernel}
p_t(\omega,\eta)%
= \frac{1}{2}(1+\euler^{-2t})\ind_{\{\omega=\eta \neq 0\}}%
+ \frac{1}{2}(1-\euler^{-2t})\ind_{\{\omega\eta =-1\}}%
+ \ind_{\{\omega=\eta=0\}}%
\qquad \forall \omega,\eta\in \{-1,0,1\}.%
\end{equation}
This time evolution acts like a heating of the system where the distribution and quantity of occupied sites does not change. For a given initial Gibbs measure $\mu$ on $(\Omega,\mathcal{F})$ at time $t=0$, we then define the time-evolved measure as the action of this semigroup on the initial measure $\mu_t := \mu \pi_t$. The expectation of a local function $f$ under the time-evolved measure is thus given by
\begin{equation}
\label{eq:local_expectation}
\mu_t(f) = \int_{\omega_{\Lambda}\in\Omega_{\Lambda}}%
\int_{\eta_{\Lambda}\in\Omega_{\Lambda}}%
f(\eta)\prod_{i\in\Lambda}p_t(\omega_i,d\eta_i)\mu(d\omega)%
\end{equation}
where $\Lambda$ is the support of $f$. We are interested in the properties of this time-evolved measure, specifically to what extent it still admits to a Gibbsian description. By definition $\mu_t$ would be a Gibbs measure if it is compatible with a quasilocal non-null specification. To contradict this property it therefore suffices to show the existence of a non-removable point of discontinuity for the expected value $\mu_t(f|\mathcal{F}_{\!\comp{\Lambda}})$ of one local function $f$, as this discontinuity would persist in each specification admitting the measure $\mu_t$, thereby showing the non-quasilocality of all compatible specifications. Such points of non-removable discontinuity will be called \emph{bad configurations}, configurations which are not bad are called $\emph{good configurations}$. A configuration $\eta$ is a bad configuration if the measure is essentially-discontinuous at $\eta$ (cp. \cite{fernandez_gibbsianness_2006}):
\begin{definition}
\label{def:essential_discontinuity}
The measure $\mu_t$ is essentially-discontinuous at $\eta\in \Omega$ if there exists a local function $f$ and a region $\Lambda_0\Subset V$ such that
\begin{equation}
\label{eq:essential_discontinuity}
\limsup_{\Lambda\nearrow V}%
\sup_{{\substack{\xi^1,\xi^2\in\Omega \\ \Delta:\Lambda\subset\Delta\Subset V}}}%
\big| \mu_t(f|\eta_{\Lambda\setminus\Lambda_0}\xi^1_{\Delta\setminus\Lambda})%
- \mu_t(f|\eta_{\Lambda\setminus\Lambda_0}\xi^2_{\Delta\setminus\Lambda})\big| > 0.%
\end{equation}
\end{definition}
If $\eta$ is an essential discontinuity in the sense of Definition \ref{def:essential_discontinuity}, it must be a discontinuity point in the product topology for every specification $\gamma$ for which $\mu$ is a compatible measure, which can be quickly seen as follows: Take an arbitrary compatible specification $\gamma$. Then, one  may write the terms with finite-volume conditionings of the form $\mu_t(f|\eta_{\Delta\setminus \Lambda_0})$ appearing in Definition \ref{def:essential_discontinuity} as an integral of $\gamma_{\Lambda_0}(f|\eta_{\Delta\setminus \Lambda_0}\zeta_{\Delta^c})$ over the variables $\zeta_{\Delta^c}$ with respect to some conditional measure which is not important for the argument. Using uniform upper and lower bounds on $\zeta_{\Delta^c}$, this shows that the left hand side in \eqref{eq:essential_discontinuity} is a \textit{lower} bound for the analogous expression involving the kernel $\gamma_{\Lambda_0}$, namely
\begin{equation}
\limsup_{\Lambda \nearrow V} \sup_{\xi^1,\xi^2 \in \Omega}%
\Big(\gamma_{\Lambda_0}(f | \eta_{\Lambda\setminus \Lambda_0}\xi^1_{V \setminus \Lambda})%
-\gamma_{\Lambda_0}(f | \eta_{\Lambda\setminus \Lambda_0}\xi^2_{V \setminus \Lambda})\Big),%
\end{equation}
which therefore is strictly positive, too.

\subsection{Results: Badness}
\label{subsec:results_badness}

\begin{theorem}
\label{thm:almost-sure-badness}
Let $(V\!,E)$ be the Cayley tree of order $d \geq 4$. Then there exist finite positive constants $\beta_{b}(d)>0$, $\lambda_{b}(d)>0$ such that for all $\beta>\beta_{b}(d)$ there exists a finite time $t_{b}(\beta,d)$ so that the set of bad configurations for $\mu^\#_{\beta,\lambda,t}$ has full measure for all $t \geq t_{b}(\beta,d)$ and all $\lambda \geq \lambda_{b}(d)$.
\end{theorem}
The proof of the theorem relies on a general zero-one law, Theorem \ref{thm:zero-one_law}, together with two ingredients about the set of a bad configurations which are interesting in themselves. A subtree condition for bad configurations, Theorem \ref{thm:condition_bad_configurations}, and the typicality of this condition at large activities, Proposition \ref{prop:probability_subtrees}.

\subsubsection*{Zero-one law}

First we give a general zero-one theorem for the set of bad configurations on trees, for possibly non-extremal measures. Note that for extremal Gibbs measures on any countable graph, it is well known that the set of bad configurations has probability zero or one. This is clear, as the set of bad configurations form a tail-event, and extremality of a Gibbs measure implies its triviality on the tail-sigma algebra. In our case however, the intermediate measure under consideration is provably non-extremal in the interesting regime of large repulsion $\beta$, and large activity $\lambda$, as we show in Lemma \ref{lemma:non-extremality_intermediate}, and so the following theorem is necessary:
\begin{theorem}[Zero-one law]
\label{thm:zero-one_law}
Assume that $\mu$ is a tree-indexed Markov chain on the Cayley tree with state space $\{-1,0,1\}^V$, which is invariant under tree-automorphisms, and whose transition matrix $P$ has strictly positive matrix elements. Denote by $\mu_t$ the corresponding time-evolved measure with starting measure $\mu$, obtained under the spin-flip dynamics \eqref{eq:single_site_kernel}. Then, for each time $t$, the set of bad configurations $B_t$ for the time-evolved measure $\mu_{t}$ satisfies the zero-one law $\mu_t(B_t)\in \{0,1 \}$.
\end{theorem}

\subsubsection*{Subtree condition on badness}

The following theorem gives a sufficient condition for bad configurations of the time-evolved measure, uniformly in the choice of signs. 
\begin{theorem}
\label{thm:condition_bad_configurations}
Let $(V\!,E)$ be the Cayley tree of order $d$ and $\eta\in\Omega$ any configuration such that the set of occupied sites $\mathcal{O}(\eta):=\{i\in V|\,|\eta_i|=1\}$ \emph{contains} a rooted tree of order $s$, where $s$ satisfies
\begin{equation}
s>\frac{d+1}{2}.%
\end{equation}
Then there exists a critical repulsion strength $\beta_{c}(d,s) \in (0,\infty)$ so that for all $\beta > \beta_{c}(d,s)$ there exists a time $t_{c}(\beta,d,s)\in (0,\infty)$ such that the time-evolved intermediate measure $\mu^{\#}_{\beta,\lambda,t}$ is essentially-discontinuous at $\eta$ for all times $t \geq t_{c}(\beta,d,s)$ and all activities $\lambda > 0$.
\end{theorem}
\begin{remark}
\label{remark:non-gibbs}
Applying Theorem \ref{thm:condition_bad_configurations} for $d=s\geq 2$ we immediately obtain that the time-evolved intermediate measure is non-Gibbs at large $\beta$ for all sufficiently large times for \emph{any} $d \geq 2$. This is clear as we are provided with the bad configurations constructed from s-subtrees (which in general may however have zero measure). 
\end{remark}

\subsubsection*{Subtree percolation}  

When is the set of provably bad configurations from Theorem \ref{thm:condition_bad_configurations} typical, i.e. when is subtree-percolation ensured? The measure of occupied sites drawn from $\smash[b]{\mu^{\#}_{\beta,\lambda,t}}$ turns out to be a tree-indexed Markov chain again, see Lemma \ref{lemma:occupation_markov_chain}, with an explicit transition matrix depending on $\beta,\lambda$ but independent of $t$ -- which should not be expected from the other Gibbs measures. Hence the connected clusters of occupied sites of the intermediate measure (growing away from the origin, see proof of zero-one law) form Galton-Watson processes that do not depend on $t$.

Let $p_{s}(\beta,\lambda,d)$ denote the time-independent probability that a fixed occupied site on the Cayley tree of order $d$, whose sites are occupied according to the time-evolved intermediate measure $\mu^{\#}_{\beta,\lambda,t}$, is the root of an outward growing occupied subtree where each vertex has at least $s$ children.
\begin{proposition}
\label{prop:probability_subtrees}
For $d \geq 2$ there exists a critical activity $\lambda_{b}(d) \in (0,\infty)$ such that
\begin{equation}
p_s(\beta,\lambda,d) > 0%
\end{equation}
holds for all $\lambda \geq \lambda_{b}(d)$ uniformly for all $\beta > 0$ and any $s \leq d-1$.
\end{proposition}
In \cite{pakes_family_1991} probabilities for occupied subtrees have already been studied, also \cite{balogh_bootstrap_2006} investigates the probability for the existence of so-called $k$-forts which implies the existence of an occupied subtree of order $d-k$. However, both of these more general works do not immediately give the bounds for our special case, thus we give a self-contained proof in Subsection \ref{subsec:subtree_percolation}.
\begin{remark}
One may ask if the full-measure badness of Theorem \ref{thm:almost-sure-badness} could persist on the ternary ($d=3$) or even on a binary tree ($d=2$), and try to improve the badness condition in Theorem \ref{thm:condition_bad_configurations}. A more refined approach for the recursion might allow better estimates by using typical influences of spins on the occupied subtree and the non-percolating branches, instead of estimating uniformly by the worst possible cases (cp. Subsection \ref{subsec:subtree_badness}). Furthermore we might obtain badness results for occupied subtrees closer to full occupation than the subtree of order $d-1$, yet still typical in a region of very large activity, thereby improving the current proof of Theorem \ref{thm:almost-sure-badness}. This remains an open problem which needs  a finer analysis.
\end{remark}

\subsection{Results: Goodness}
\label{subsec:results_goodness}

\subsubsection*{General result: short-time goodness via Dobrushin}
As we have seen the time-evolved measure $\smash[b]{\mu^{\#}_{\beta,\lambda,t}}$ is not Gibbs for large enough times $t$ and activities $\lambda$. However, there are regimes of parameters where $\smash[b]{\mu^{\#}_{\beta,\lambda,t}}$ is Gibbs or at least almost surely Gibbs. First we state that the intermediate dynamical measure satisfies the so-called short-time Gibbs property, i.e. $\mu^{\#}_{\beta,\lambda,t}$ is Gibbs for small times $t$.
\begin{theorem}
\label{thm:almost-sure-goodness_dobrushin}
For every $\beta>0$ and $\lambda>0$ there exists a time $t_g(\beta,\lambda,d)\in(0,\infty]$ such that for all $t<t_g(\beta,\lambda,d)$ the time-evolved measure $\mu^{\#}_{\beta,\lambda,t}$ is Gibbs.
\end{theorem}
\begin{remark}
By this theorem and Remark \ref{remark:non-gibbs} we found a Gibbs-non-Gibbs transition for the time-evolved intermediate measure on the Cayley tree of order $d$ at large repulsion strength $\beta$. For all activities $\lambda$ the measure is Gibbs for small times and non-Gibbs for large times.
\end{remark}

\subsubsection*{Almost sure goodness for small density via extinction}
In contrast to Theorem \ref{thm:almost-sure-badness} the set of bad configurations has zero measure, since it is empty, for small times and every activity. In the following theorem we handle the case for small $\lambda$. Here we prove only an almost-sure result, i.e. the set of bad configurations has zero measure. The idea to rewrite the model to use extinction probabilities for Galton-Watson trees.
\begin{theorem}
\label{thm:almost-sure-goodness}
Let $\beta>0$. Then there exists $\lambda_{g}(\beta,d) \in (0,\infty)$ such that for every $\lambda < \lambda_{g}(\beta,d)$ the time-evolved measure $\mu^{\#}_{\beta,\lambda,t}$ is almost surely Gibbs for every time $t>0$.
\end{theorem}
\begin{remark}
From Theorems \ref{thm:almost-sure-badness} and \ref{thm:almost-sure-goodness} we get an $\lambda$-dependent transition for the set of bad configurations of the time-evolved intermediate measure on Cayley trees of order $d \geq 4$ from measure zero to measure one. For $\beta > \beta_{b}(d)$ and $t \geq t_{b}(\beta,d)$ the set of bad configurations has measure zero for activities $\lambda < \lambda_{g}(\beta,d)$ and measure one for large activities $\lambda \geq \lambda_{b}(d)$.
\end{remark}

\section{Proofs: Badness}
\label{sec:proofs_badness}

\subsection{Renewal subtree construction and zero-one law for bad configurations}	
\label{subsec:zero-one}

We now give the proof of the zero-one law for bad configurations.
\begin{proof}[Proof of Theorem \ref{thm:zero-one_law}]
The main idea is to describe a configuration drawn from the tree-indexed Markov chain $\mu$ in terms of i.i.d. building blocks which are given by the connected components of occupied sites, along with their signs, anchored at the site on the cluster closest to an origin. Using only the tree-indexed Markov chain property of $\mu$, the clusters will be grown by means of a recursive algorithm described below. Here the anchoring sites of the clusters will appear as \textit{so-called active sites}. They will be determined depending on the clusters which have grown in the previous steps. These anchored clusters can be viewed as geometric generalizations to the excursions of a stationary Markov chain. 

More precisely,  we start by enumerating the vertices of the tree by $\mathbb{N}\cup \{0\}$: Choose an arbitrary root of the vertex set of the tree and label it by $0$. Then label the $d+1$ sites at distance $1$ to the root by the integers $1, \dots, d+1$, in otherwise arbitrary order. Next label the sites at distance $2$ to the root by the $(d+1)d$ next integers, in otherwise arbitrary order. Next label the sites at distance $3$ to the root by the next integers, in otherwise arbitrary order. Proceed in this way for all finite distances.

\subsubsection*{Spiralling renewal}

\begin{figure}[t]
\centering%
\includegraphics[width=0.7\textwidth]{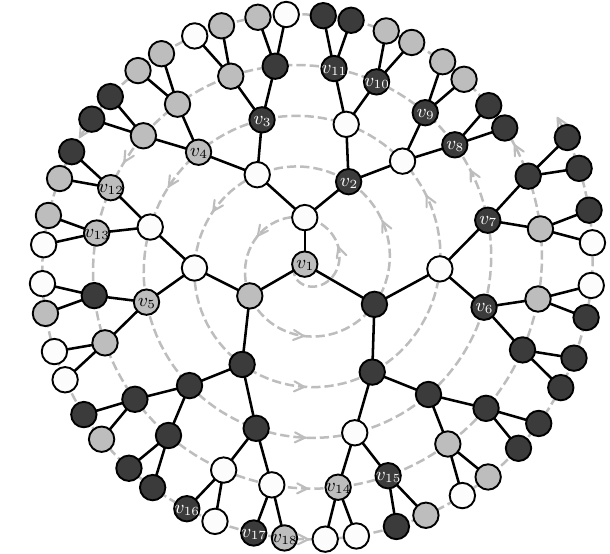}%
\caption{The vertices of a tree can be enumerated along a spiral emanating from an arbitrary root index. The occupied vertices, here coloured in grey and black for plus and minus spins, can be seen as active sites growing an outward pointing (away from the root) tree with stopping symbol $0$ provided they are not the child of another occupied vertex. Enumerating these active sites along the spiral we obtain an almost surely infinite sequence of active sites $(v_i)_{i \in \mathbb{N}}$. The partial trees growing from these active sites are i.i.d..}%
\label{fig:spiralling}%
\end{figure}

We construct a configuration of $\mu$ according to the following spiralling renewal-construction, based on the tree-indexed Markov chain property:

\emph{Step 0 -- Initialisation:} Choose the value $\sigma_0$ at the root according to the single-site distribution $\rho$, which is the invariant distribution for the transition matrix $P$ of the Markov-chain Gibbs measure, i.e. it satisfies $\rho=\rho P$. Note that our assumption on $P$ implies strict positivity of the entries of $\rho$. In the case $\sigma_0=s$ where $s\in \{-1,1\}$, choose the active site as $v=0$ and turn to the next step, \emph{Step 1}. In the case $\sigma_0=0$ choose the value of $\sigma_1=s_1$ with probability $P(0,s_1)>0$. In the case $\sigma_1\neq 0$, call $1$ the active site and turn to the next step. Otherwise carry on this procedure, i.e. continue according to the enumeration of sites in this way until the first vertex $v$ is reached for which $\sigma_v \neq 0$. Call this vertex $v$ the \emph{active site} for \emph{Step 1}. 

\emph{Step 1 -- Grow partial trees from an active site $v$, with stopping symbol $0$:} In the case $\sigma_v=s$ where $s\in \{-1,1\}$, grow random partial configurations $\sigma^s_{V_v}$ on the random subset $V_v$, which contains $v$ and points to the outside, in the following way. Here $V_v$ by definition should contain the active site $v$ as its smallest site according to the enumeration. Apply the transition matrix $P$, from inside to outside, away from the root, starting with initial condition $\sigma_v=s$. Stop to grow the branches to the outside when a first zero appears in the branch, and keep these first zeros together with the configurations of zeros and signs obtained so far. This has filled a (finite or infinite) part $V_v$ of the tree emerging from site $v$, where the value $0$ plays the role of a stopping symbol which marks the boundary of the connected component.

We note that for $v \neq 0$ the distribution of $\sigma^s_{V_v}$ is tree-invariant, and describes the connected component attached at the site $v$. For the particular case $v=0$, maximally $d+1$ such components come together to form the component at the origin.

\emph{Step 2 -- Filling more zeros to get to the next active site:} Determine the smallest site $z$ according to the spiralling enumeration whose spin value has not yet been determined. By construction the spin value of the parent site of $z$ on the tree was already determined to be zero. Choose the value $\sigma_z=s_z$ with probability according to $P(0,s_z)$. In the case $\sigma_z=0$ repeat \emph{Step 2}. In the case $\sigma_z=s$ for $s\in \{-1,1\}$ go to \emph{Step 1} with new active site $z$.

This procedure produces a sequence of active sites $(v_i)_{i\in \mathbb{N}}$, as shown in Figure \ref{fig:spiralling}, along with their signs $s_i$. As the transition matrix $P$ has strictly positive matrix elements, this sequence is almost surely infinite. We then obtain a configuration $\sigma_V$ on the full vertex set of the tree which is distributed according to $\mu$ as the concatenation of $(\sigma^{s_i}_{V_{v_i}})_{i\in \mathbb{N}} $ with the empty configuration $0$ on the remaining sites $(\bigcup_{i\in \mathbb{N}}V_{v_i})^c$. We note that (up to the component of the root) the components $\sigma^{s_i}_{V_{v_i}}$ are (tree-isomorphic to) i.i.d. random objects. \\

Now we come to the proof of the zero-one property of bad configurations. A configuration $\sigma_V$ on the full tree is a bad configuration if and only if there is at least one connected component of its occupied sites which acts as a bad configuration (when the configuration is continued by zero outside). This follows as the first-layer model decouples over the connected components of occupied sites in the conditioning. 

We can therefore check badness of the time-evolved measure on the connected components grown from $v_i\neq 0$ with sign $s\in \{-1,1\}$. We note for this purpose that the probabilities
\begin{equation}
p_{bad}(s):=\bar\mu(\sigma^s_{V_{v_i}}0_{(V_{v_i})^c} \text{ is bad at time } t\,|\,s_i=s)%
\end{equation}
do not depend on $i$ for $i \geq 2$, where the measure $\bar \mu$ is obtained as the representation of the measure $\mu$ via the above renewal construction, together with the semi-group of the time evolution. This is clear by the invariance of the construction of $V_{v_i}$, and the tree-automorphism invariance of the property to be a bad configuration.

\emph{Case 1:} $p_{bad}(s')>0$ for at least one spin value $s'\in \{-1,1\}$. As the sequence of active sites takes the value $v_i=s'$ for infinitely many $i \geq 2$ with probability one, we conclude that $\mu$-almost surely there are even infinitely many connected components labelled by $i \geq 2$ which are bad. This follows by the Borel-Cantelli Lemma applied to the situation of independently many trials with positive probability.  In particular we have $\mu_t(B_t)=1$ for the set of bad configurations $B_t$.

\emph{Case 2:} $p_{bad}(1)=p_{bad}(-1)=0$. Then all the components for $i \geq 2$ carry good configurations almost surely. We remark that also for the first component we have that $\bar\mu(\sigma^s_{V_{v_1}}0_{(V_{v_1})^c} \text{ is bad at time } t|s_1=s)=0$, for both values of $s$, if all connected components away from the origin are good. To see this, condition on the event that $v_1=0$ is in fact the origin, and employ a small Gibbsian computation which shows that glueing of finitely many good configurations on connected components preserves the property to be a good configuration. Hence it follows $\mu_t(B_t)=0$.
\end{proof}

\subsection{Subtree condition on badness}
\label{subsec:subtree_badness}

First, we introduce a sufficient criterion for essential discontinuity of the time-evolved intermediate measure. We treat the time evolution kernels as additional fields on the first-layer model, the percolation of information in the time-evolved model can then be expressed as a recursion depending on these fields.
\begin{lemma}
\label{lemma:essential_discontinuity_bf}
The measure $\mu^{\#}_{\beta,\lambda,t}$ is essentially-discontinuous at configuration $\eta\in\Omega$ if there exists a vertex $0 \in V$, an $\epsilon > 0$, and cofinal sequences $(\Lambda_n)_{n \in \mathbb{N}} \nearrow V$, $(\Delta_n)_{n \in \mathbb{N}} \nearrow V$ with $0 \in \Lambda_n \subset \Delta_n \Subset V$ such that
\begin{equation}
\label{eq:essential_discontinuity_bf}
\lim_{n \rightarrow \infty}\left|f_{k0}[\eta_{\Lambda_n\setminus 0}+_{\Delta_n\setminus\Lambda_n}]%
- f_{k0}[\eta_{\Lambda_n\setminus 0}-_{\Delta_n\setminus\Lambda_n}]\right| \geq \epsilon,%
\end{equation}
for at least one \emph{occupied} vertex $k \in \boundary 0$. The $f_{ij}[\xi_{A \setminus 0}]$ are \emph{boundary fields} depending on the configuration $\xi$ in $A \setminus 0$ calculated for edges pointing towards $0$ through the recursion
\begin{equation}
\label{eq:boundary_field_recursion}
f_{ij}[\xi_{A\setminus 0}]%
= \sum_{\substack{k \in \boundary{i}\setminus j \\ |\xi_k| = 1}}%
\varphi_{\beta/2}(f_{ki}[\xi_{A\setminus 0}]+h^t\xi_k)
\end{equation}
with
\begin{align}
\label{eq:definition_phi}
\varphi_{\beta}(x) &:= \frac{1}{2}\log\left(%
\frac{\cosh\left(x+\beta\right)}{\cosh\left(x-\beta\right)}\right) \\%
\label{eq:time-factor}
h^t &:= \frac{1}{2}\log\frac{1+\euler^{-2t}}{1-\euler^{-2t}}%
\end{align}
and homogeneous starting values $f_{ij}[\xi_{A \setminus 0}]=0$ at the boundary of $A$, i.e. for all edges $\textedge{ij}$ pointing towards $0$ with $i \in \boundary{A}$, independent of the configuration $\xi$.
\end{lemma}
Using the definition of essential discontinuity \eqref{eq:essential_discontinuity}, this lemma will be proved by representing the single-site probabilities of the time-evolved measure $\mu^{\#}_{\beta,\lambda,t}$ conditioned on a finite neighbourhood as a sum over compatible first-layer configurations, i.e. those configurations at time $t=0$ which could eventually evolve to the second-layer configuration in the conditioning. These first-layer configurations have a closed representation through boundary laws, as the initial measure $\mu^{\#}_{\beta,\lambda,0}$ at time $t=0$ is a tree-indexed Markov chain. Each summand is weighted by a modified Hamiltonian that includes field-like terms originating from the time evolution. Executing the sum then leads to the recursion relation above and a representation of the conditioned single-site probabilities through a first-layer Hamiltonian with the additional boundary fields that globally depend on the second-layer configuration. Therefore a discontinuity in these fields translates to essential discontinuity of the measure.
\begin{proof}[Proof of Lemma \ref{lemma:essential_discontinuity_bf}]
First note that we can naturally write the conditional probability of the time-evolved intermediate measure $\mu^{\#}_{\beta,\lambda,t}$ at a fixed but arbitrary root index $0$ conditioned on a second-layer spin configuration $\eta$ in a finite neighbourhood $\Lambda\setminus 0$ as
\begin{equation}
\label{eq:second_layer_representation}
\mu^{\#}_{\beta,\lambda,t}(\sigma_0=\eta_0|\eta_{\Lambda\setminus 0})%
= \int \hat{\mu}^{\#}_{\beta,\lambda,t}[\eta_{\Lambda\setminus 0}]%
(d\omega_0)p_t(\omega_0,\eta_0),%
\end{equation}
where $\hat{\mu}^{\#}_{\beta,\lambda,t}$ is the probability of the first-layer spin value at $0$ conditioned on the given second-layer configuration $\eta$, which by \eqref{eq:local_expectation} has the representation
\begin{equation}
\label{eq:second_layer_cond_probability}
\hat{\mu}^{\#}_{\beta,\lambda,t}[\eta_{\Lambda\setminus 0}](\omega_0)%
:= \frac{1}{Z_{\beta,\lambda,t}[\eta_{\Lambda\setminus 0}]}%
\sum_{\omega'_{\Lambda\setminus 0}\in \Omega_{\Lambda\setminus 0}}%
\mu^{\#}_{\beta,\lambda,0}\big(\sigma_{\Lambda}=\omega_0\omega'_{\Lambda\setminus 0}\big)%
\prod_{i\in\Lambda\setminus 0}p_t(\omega'_i,\eta_i).%
\end{equation}
Here $\mu^{\#}_{\beta,\lambda,0}$ denotes the intermediate measure at time $t=0$ and $Z_{\beta,\lambda,t}[\eta_{\Lambda\setminus 0}]$ is a suitable normalisation -- we will use this notation for normalisations of different expressions without further apology.

Since $p_t$ interpreted as a matrix is bijective, from Equation \eqref{eq:second_layer_representation} follows, that we can infer essential discontinuity of $\mu^{\#}_{\beta,\lambda,t}$ by proving that the family of measures $\big(\hat{\mu}^{\#}_{\beta,\lambda,t}[\eta_{\Lambda\setminus 0}]\big)_{\eta\in\Omega,\Lambda\Subset\Omega}$ fulfils the condition
\begin{equation}
\label{eq:essential_discontinuity_hash}
\lim_{n \rightarrow \infty}%
\Big| \hat{\mu}^{\#}_{\beta,\lambda,t}[\eta_{\Lambda_n\setminus 0}+_{\Delta_n\setminus\Lambda_n}](\omega_0)%
- \hat{\mu}^{\#}_{\beta,\lambda,t}[\eta_{\Lambda_n\setminus 0}-_{\Delta_n\setminus\Lambda_n}](\omega_0)\Big|%
\geq \epsilon'%
\end{equation}
for a fixed $\epsilon' > 0$ and cofinal sequences $(\Lambda_n)_{n \in \mathbb{N}} \nearrow V$, $(\Delta_n)_{n \in \mathbb{N}} \nearrow V$ with $0 \in \Lambda_n \subset \Delta_n \Subset V$ and some $\omega_0 \in S$. Here we have chosen $\ind_{\{\omega_0\}}$ as local function and $+,\,-$ denote the fixed configurations that are plus and minus everywhere on the whole tree respectively.

To rewrite $\mu^{\#}_{\beta,\lambda,t}$ in terms of boundary fields we first note that the single-site time evolution \eqref{eq:single_site_kernel} can be rewritten in an exponential form
\begin{equation}
p_t(\omega_i,\eta_i) = c^t(\omega_i,\eta_i)\exp\big(h^t\omega_i\eta_i\big)%
\end{equation}
using $h^t$ defined in \eqref{eq:time-factor} and
\begin{equation}
c^t(\omega_i,\eta_i)%
= \left\{%
\begin{array}{ccc}
\frac{1}{2}\big(1-\euler^{-4t}\big)^{\frac{1}{2}} && \mathrm{if}\ |\omega_i| = |\eta_i| = 1 \\%
1 && \mathrm{if}\ |\omega_i| = |\eta_i| = 0 \\%
0 && \mathrm{else}%
\end{array}\right..%
\end{equation}
This time evolution prohibits first-layer configurations in the sum of Equation \eqref{eq:second_layer_cond_probability} whose set of occupied sites differs from that of the prescribed second-layer configuration $\eta_{\Lambda\setminus 0}$. We therefore introduce for each second-layer configuration $\eta$ the space of its \emph{compatible configurations} in the first-layer $\Omega^{\eta}$ by
\begin{equation}
\Omega^{\eta} := \big\{\omega\in\Omega\,\big|\,|\omega_i|=|\eta_i| \quad \forall i\in V\big\}%
\end{equation}
and can restrict the sum to finite-volume configurations of $\Omega^{\eta}$. Replacing the single-site kernels through their exponential notation then leads to
\begin{equation}
\label{eq:muhat_first_layer}
\hat{\mu}^{\#}_{\beta,\lambda,t}[\eta_{\Lambda\setminus 0}](\omega_0)%
= \frac{1}{Z_{\beta,\lambda,t}[\eta_{\Lambda\setminus 0}]}%
\sum_{\omega'_{\Lambda\setminus 0}\in \Omega^{\eta}_{\Lambda\setminus 0}}%
\mu^{\#}_{\beta,\lambda,0}\big(\sigma_{\Lambda}=\omega_0\omega'_{\Lambda\setminus 0}\big)%
\,\prod_{\mathclap{i\in\Lambda\setminus 0}}\,%
\exp\big(h^t\omega'_i\eta_i\big).%
\end{equation}
Here we were allowed to include the factors $c^t$ of the time evolution in the normalisation as they are identical for all compatible configurations in $\smash[b]{\Omega^{\eta}_{\Lambda\setminus 0}}$.

Since the initial measure $\mu^{\#}_{\beta,\lambda,0}$ is a tree-indexed Markov chain, it has a representation through boundary laws (cp. Theorem \ref{thm:bl_gibbs_measure})
\begin{equation}
\label{eq:boundary_law_representation_hash}
\mu^{\#}_{\beta,\lambda,0}(\sigma_{\Lambda\cup\boundary{\Lambda}}=\omega_{\Lambda\cup\boundary{\Lambda}})%
= \frac{1}{Z''_{\Lambda;\beta,\lambda}}%
\prod_{\substack{\{i,j\} \in E \\ \{i,j\}\cap\Lambda\neq\emptyset}}%
Q_{\{i,j\}}(\omega_i,\omega_j)%
\prod_{k\in\boundary{\Lambda}}l_{kk_{\Lambda}}(\omega_k).%
\end{equation}
Replacing $\mu^{\#}_{\beta,\lambda,0}$ in \eqref{eq:muhat_first_layer} through its boundary law representation \eqref{eq:boundary_law_representation_hash} yields an expression that just consists of functions for spin-values of at most two neighbouring vertices. As trees do not contain loops this means that the sum over configurations on $\Lambda$ can be split into multiple sums over configurations on disjoint connected components of $\Lambda$, which just communicate through one unique path. If this path contains a vertex with a fixed spin-value, as is the case for every compatible first-layer configuration if the second-layer configuration $\eta$ has spin-value zero at one of the paths sites, these components are fully independent. In particular this implies that Equation \eqref{eq:muhat_first_layer} only depends on spin-values of sites which have a direct connection to the root through a path that is occupied in the second-layer configuration $\eta$. We therefore define the set of vertices connected to the root with regard to occupation in $\eta$
\begin{equation}
\mathcal{C}^{\eta}(A) := \big\{i \in A\,\big|\, |\eta_j|=1 \quad \forall j\in\mathcal{P}(0,i)\big\},%
\end{equation}
and can restrict Equation \eqref{eq:muhat_first_layer} to configurations on this connected component
\begin{equation}
\label{eq:muhat_intermission}
\hat{\mu}^{\#}_{\beta,\lambda,t}[\eta_{\Lambda\setminus 0}](\omega_0)%
= \begin{multlined}[t]%
\frac{1}{Z_{\beta,\lambda,t}[\eta_{\Lambda\setminus 0}]}%
\sum_{\omega_{\mathcal{C}^{\eta}(\Lambda\setminus 0)}\in \Omega_{\mathcal{C}^{\eta}(\Lambda\setminus 0)}}%
\prod_{\substack{\{i,j\} \in E \\ \{i,j\} \subset \mathcal{C}^{\eta}(\Lambda)}}%
Q_{\{i,j\}}(\omega_i,\omega_j) \\%
\times\prod_{k\in \mathcal{C}^{\eta}(\boundary{\Lambda})} l_{kk_{\Lambda}}(\omega_k)%
\prod_{l\in \mathcal{C}^{\eta}(\Lambda\setminus 0)}\exp\big(h^t\omega_l\eta_l\big).%
\end{multlined}%
\end{equation}
Except for the root, all sites in Equation \eqref{eq:muhat_intermission} have spins that are guaranteed to be non-zero. As the boundary laws of the intermediate measure $\mu^{\#}_{\beta,\lambda,0}$ are identical for occupied vertices they can be included in the normalisation resulting in
\begin{equation}
\label{eq_muhat_intermission_two}
\begin{multlined}[t]%
\hat{\mu}^{\#}_{\beta,\lambda,t}[\eta_{\Lambda\setminus 0}](\omega_0)%
= \frac{1}{Z_{\beta,\lambda,t}[\eta_{\Lambda\setminus 0}]} \\%
\times\sum_{\omega_{\mathcal{C}^{\eta}(\Lambda\setminus 0)}\in\{-1,1\}^{\mathcal{C}^{\eta}(\Lambda\setminus 0)}}%
\prod_{\substack{\{i,j\} \in E \\ \{i,j\}\subset \mathcal{C}^{\eta}(\Lambda)}}%
Q_{\{i,j\}}(\omega_i,\omega_j)%
\exp\Bigg(\sum_{k\in \mathcal{C}^{\eta}(\Lambda \setminus 0)}h^t\eta_k\omega_k%
\Bigg).\end{multlined}%
\end{equation}
Executing the summation over the first-layer spin values in \eqref{eq_muhat_intermission_two} can be done successively for each site, beginning at the boundary and working inwards to the root. By defining \emph{boundary fields} via the recursion
\begin{equation}
\label{eq:boundary_field_recursion_2}
f_{ij}[\eta_{\Lambda\setminus 0}]%
:= \sum_{\substack{k \in \boundary{i}\setminus j \\ |\eta_k| = 1}}%
\varphi_{\beta/2}(f_{ki}[\eta_{\Lambda\setminus 0}]+h^t\eta_k),%
\end{equation}
where $\varphi_{\beta}$ has already been defined in \eqref{eq:definition_phi}, with homogeneous starting values $f_{ij}[\eta_{\Lambda\setminus 0}] = 0$ for all $i \in \boundary{\Lambda}$, we eventually get the representation
\begin{equation}
\label{eq:hat_representation}
\begin{multlined}[t]
\hat{\mu}^{\#}_{\beta,\lambda,t}[\eta_{\Lambda\setminus 0}](\omega_0)%
= \frac{\lambda^{|\omega_0|}}{Z_{\beta,\lambda,t}[\eta_{\Lambda\setminus 0}]} \\%
\times\sum_{\omega_{\mathcal{C}^{\eta}(\boundary{0})}\in\{-1,1\}^{\mathcal{C}^{\eta}(\boundary{0})}}%
\exp\Bigg(%
\sum_{\substack{k\in \boundary{0} \\ |\eta_k| = 1}}\Big(%
\beta\ind_{\{\omega_k\omega_0=-1\}}%
+ h^t\eta_k\omega_k%
+ f_{k0}[\eta_{\Lambda\setminus 0}]\omega_k%
\Big)\Bigg).%
\end{multlined}
\end{equation}
We note that the activity $\lambda$ plays no part in the recursion process, as all spins at vertices on $\mathcal{C}^{\eta}(\Lambda \setminus 0)$ are occupied, which allows us to incorporate the $\lambda$-dependent parts of the transfer operators in the normalisation.

The only components of Equation \eqref{eq:hat_representation} which are dependent on the global behaviour of $\eta$ are the boundary fields $f_{k0}[\eta_{\Lambda\setminus 0}]$ calculated through the $\eta$-dependent recursion \eqref{eq:boundary_field_recursion_2}. To show that Inequality \eqref{eq:essential_discontinuity_hash} holds, it suffices to look at the difference between the boundary fields for both expressions and show
\begin{equation}
\label{eq:essential_discontinuity_sum}
\lim_{n \rightarrow \infty}%
\bigg|\sum_{\substack{k\in \boundary{0} \\ |\eta_k| = 1}}%
f_{k0}[\eta_{\Lambda_n\setminus 0}+_{\Delta_n\setminus\Lambda_n}]%
- \sum_{\substack{k\in \boundary{0} \\ |\eta_k| = 1}}%
f_{k0}[\eta_{\Lambda_n\setminus 0}-_{\Delta_n\setminus\Lambda_n}]\bigg| \geq \epsilon''%
\end{equation}
for some $\epsilon''>0$. This can be seen by an additional recursion step which results in
\begin{equation}
\hat{\mu}^{\#}_{\beta,\lambda,t}[\eta_{\Lambda \setminus 0}](\omega_0)%
= \frac{\lambda^{|\omega_0|}}{Z_{\beta,\lambda,t}[\eta_{\Lambda \setminus 0}]}%
\exp\bigg(\frac{\omega_0}{|\mathcal{C}^{\eta}(\boundary{0})|}%
\sum_{\substack{k\in \boundary{0} \\ |\eta_k| = 1}}%
f_{0k}[\eta_{\Lambda \setminus 0}]\bigg)%
\end{equation}
with new boundary fields $f_{0k}$ pointing away from the origin that due to the strict monotonicity of \eqref{eq:boundary_field_recursion_2} retain condition \eqref{eq:essential_discontinuity_sum}. The benefit of this representation lies in the fact that the influence of the boundary fields is more immediate. To see that jumps in the boundary fields carry over to jumps in the probabilities compare now the fraction
\begin{equation}
\frac{\hat{\mu}^{\#}_{\beta,\lambda,t}[\eta_{\Lambda_n\setminus 0}+_{\Delta_n\setminus\Lambda_n}](+1)}%
{\hat{\mu}^{\#}_{\beta,\lambda,t}[\eta_{\Lambda_n\setminus 0}+_{\Delta_n\setminus\Lambda_n}](-1)}%
\end{equation}
to the corresponding fraction for $\eta_{\Lambda_n\setminus 0}-_{\Delta_n\setminus\Lambda_n}$. As the recursion \eqref{eq:boundary_field_recursion_2} also preserves monotonicity in the configurations we get $f_{ij}[\eta_{\Lambda_n\setminus 0}+_{\Delta_n\setminus\Lambda_n}] \geq f_{ij}[\eta_{\Lambda_n\setminus 0}-_{\Delta_n\setminus\Lambda_n}]$ for each edge $\textedge{ij}\in\textedge{E}$. To show \eqref{eq:essential_discontinuity_sum} it is therefore sufficient to find just a single occupied vertex $k \in \boundary{0}$ such that $f_{k0}[\eta_{\Lambda_n\setminus 0}+_{\Delta_n\setminus\Lambda_n}] > f_{k0}[\eta_{\Lambda_n\setminus 0}-_{\Delta_n\setminus\Lambda_n}]$. This concludes the proof of the lemma.
\end{proof}
Having developed the essential discontinuity criterion above, let us now show that it is fulfilled for large repulsion $\beta$ at large times, if the set of occupied sites $\mathcal{O}(\eta) = \{i\in V|\,|\eta_i|=1\}$ of the configuration $\eta$ \emph{contains} a rooted subtree $S$ with $s$ children where $s$ satisfies
\begin{equation}
\label{eq:subtree_size_condition}
s > \frac{d+1}{2}.%
\end{equation}
The main idea of the following proof is depicted in Figure \ref{fig:bf_recursion}.
\begin{figure}[p]
\centering%
\includegraphics[width=0.9\textwidth]{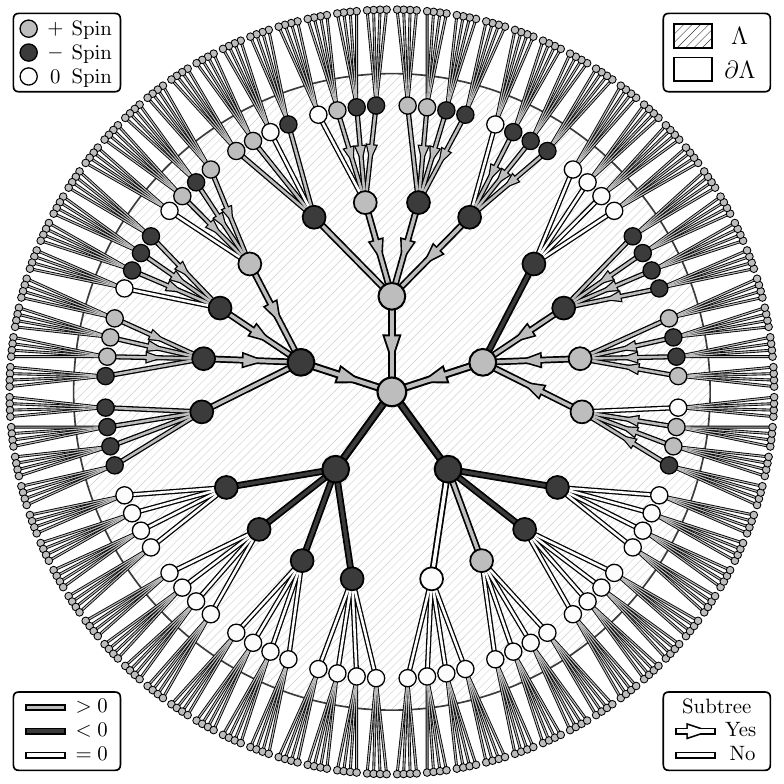}%
\caption{Simulation of the boundary field recursion on the Cayley tree of order $d=4$ at repulsion strength $\beta=2.0$ and time $t=0.2$. Vertices are coloured according to their second-layer spin values, here the configuration on $\Lambda$ is such that the centre vertex is the root of an occupied subtree with $s=3$ children. Each inward pointing edge ${ij}$ is coloured according to the value of the boundary field $f_{ij}$ plus the time-dependent field $h^t\eta_i$. The boundary field value at $\boundary{\Lambda}$ is the fixed point solution $F'>0$ of the recursion for an all plus configuration on $\Delta\setminus\Lambda$. The boundary field values on the inner rings are subsequently calculated from the recursion \eqref{eq:boundary_field_recursion} with boundary fields from unoccupied sites set to zero as they do not influence the recursion. \\ The positive starting value, combined with a large number of occupied spins due to the subtree structure, guarantees that the boundary field values along the edges of the subtree stay positive until they reach the centre vertex. This is true even for configurations where the subtree is completely occupied by minus spins. For the Ising-like case where a sub-Cayley-tree is occupied, $s=2$ children suffice for this result, regardless of the size of the main tree. In the Widom-Rowlinson model however, we might get spin clusters outside of this infinite subtree that are disconnected from the positive influences at $\boundary{\Lambda}$, like those depicted in the lower half, which can develop negative boundary field values. These need to be compensated through an in comparison sufficiently large structure of occupied spins percolating the positive boundary fields, leading to condition \eqref{eq:subtree_size_condition} for subtree percolation.}%
\label{fig:bf_recursion}%
\end{figure}
\begin{proof}[Proof of Theorem \ref{thm:condition_bad_configurations}]
\label{proof:condition_bad_configurations}
Using Lemma \ref{lemma:essential_discontinuity_bf} we choose $0$ to be the root of the occupied subtree $S$ and take $\Lambda_n=D_n,\,\Delta_n=D_{n+m}$, where $D_k:=\{i \in V|d(i,0)\leq k\}$ is the disc with radius $k$ around $0$. The value of $m \in \mathbb{N} $ will be chosen later in the proof. Using the recursion \eqref{eq:boundary_field_recursion} we can then proceed to calculate the boundary fields on the annuli $R_k:=\boundary{D_{k-1}}$ for $k \in \{1,\ldots,n+m+1\}$ (beginning at $R_{n+m+1}$ working inwards toward $\boundary{0}=R_1$).

We will do so for the plus configuration on $\Delta\setminus\Lambda$, the proof for the minus condition works analogously. Since $+_{\Delta\setminus\Lambda}$ is homogeneous, the boundary fields on each ring $R_k \subset \Delta\setminus\Lambda$ are also homogeneous. We denote these homogeneous values by $F_k$ and it holds
\begin{equation}
f_{ij}[\eta_{\Lambda\setminus 0}+_{\Delta\setminus\Lambda}] = F_k%
\qquad \forall\ \edge{ij}\ \mathrm{with}\ i\in R_k, j\in R_{k-1}.%
\end{equation}
From the boundary field recursion \eqref{eq:boundary_field_recursion} we get
\begin{equation}
\label{eq:recursion_delta}
F_k = d\varphi_{\beta/2}(F_{k+1} + h^t)%
\qquad \forall k \in \{n+1,\ldots,n+m\}%
\end{equation}
with starting value $F_{n+m+1}=0$. The function
\begin{equation}
\label{eq:recursion_function_delta}
x \mapsto d\varphi_{\beta/2}(x + h^t),%
\end{equation}
always has exactly one positive attractive fixed point $F'>0$ and we can ensure that $F_{n+1}$ is arbitrarily close to $F'$ by choosing an appropriately large $m \in \mathbb{N}$.

$F_{n+1}$ then provides the initial condition for the recursion on $\Lambda$. As the recursion on $\Lambda$ is $\eta$-dependent, we will estimate a lower bound for the boundary fields on each ring $R_k \subset \Lambda$ for the edges that are part of the subtree $S$
\begin{equation}
F_k := \min_{\substack{ij \in \edge{E},\,i,j\in S \\ i \in R_k,\, j \in R_{k-1}}}%
f_{ij}[\eta_{\Lambda\setminus 0}+_{\Delta\setminus\Lambda}]%
\qquad \forall k \in \{1,\ldots,n\},%
\end{equation}
and show that this lower bound stays positive up to the root vertex $0$. Applying the boundary field recursion \eqref{eq:boundary_field_recursion} to the $f_{ij}$ in $F_k$ gives
\begin{equation}
F_k = \min_{\substack{ij \in \edge{E},\,i,j\in S \\ i \in R_k,\, j \in R_{k-1}}}%
\sum_{\substack{l \in \boundary{i}\setminus j \\ |\eta_l| = 1}}\varphi_{\beta/2}%
\big(f_{li}[\eta_{\Lambda\setminus 0}+_{\Delta\setminus\Lambda}] + h^t\eta_l\big).%
\end{equation}
We can split the sum in terms where the vertex $l \in R_{k+1}$ is part of the subtree $S$ and those where $l$ is not contained in $S$. The boundary fields of the former terms, of which there are at least $s$, can be lower-bounded by $F_{k+1}$. Estimating $\eta_l \in \{-1,1\}$ by $-1$ and using the monotonicity of $\varphi$ then gives a lower bound for the terms on the subtree. The terms for $l \notin S$, of which there are at most $(d-s)$ can be lower-bounded by $-\beta/2$ as $|\varphi_{\beta/2}|$ is bounded by $\beta/2$. Thus we get the estimation
\begin{equation}
\label{eq:lower_bound_bf}
\begin{split}
F_k \geq& \min_{\substack{ij \in \edge{E},\,i,j\in S \\ i \in R_k,\, j \in R_{k-1}}}\bigg(%
\sum_{\substack{l \in \boundary{i}\setminus j \\ |\eta_l| = 1,\,l \in S}}%
\varphi_{\beta/2}\big(F_{k+1} - h^t\big)%
- \sum_{\substack{l \in \boundary{i}\setminus j \\ |\eta_l| = 1,\,l \notin S}}%
\frac{\beta}{2}%
\bigg)\\%
\geq& s\varphi_{\beta/2}\big(F_{k+1}%
- h^t\big) - (d-s)\frac{\beta}{2},%
\end{split}
\end{equation}
which holds for all $k \in \{1,\ldots,n\}$. A positive fixed point solution $F^+>0$ of the function
\begin{equation}
\label{eq:recursion_function_lambda}
x \mapsto s\varphi_{\beta/2}(x-h^t) - (d-s)\frac{\beta}{2},%
\end{equation}
would have to be smaller than the fixed point $F'>0$ of the outer recursion, which can be seen through direct comparison of the functions \eqref{eq:recursion_function_delta} and \eqref{eq:recursion_function_lambda}. Provided such a solution exists, we could fix $m$ to be large enough such that $F^+<F_{n+1}$. This would inductively imply
\begin{equation}
F_k%
\stackrel{\eqref{eq:lower_bound_bf}}{\geq}%
s\varphi_{\beta/2}(F_{k+1}-h^t) - (d-s)\frac{\beta}{2}%
\geq s\varphi_{\beta/2}(F^+-h^t) - (d-s)\frac{\beta}{2}%
= F^+ > 0%
\end{equation}
for all $k\in\{1,\ldots,n\}$. Here we used the monotonicity of $\varphi_{\beta/2}$ in the second step. In particular we would get $f_{j0}[\eta_{\Lambda\setminus 0}+_{\Delta\setminus\Lambda}] \geq F^+>0$ for all $j \in \boundary{0} \cap S$ and arbitrarily large $n \in \mathbb{N}$.

\begin{figure}[ht]
\centering%
\includegraphics[width=0.9\textwidth]{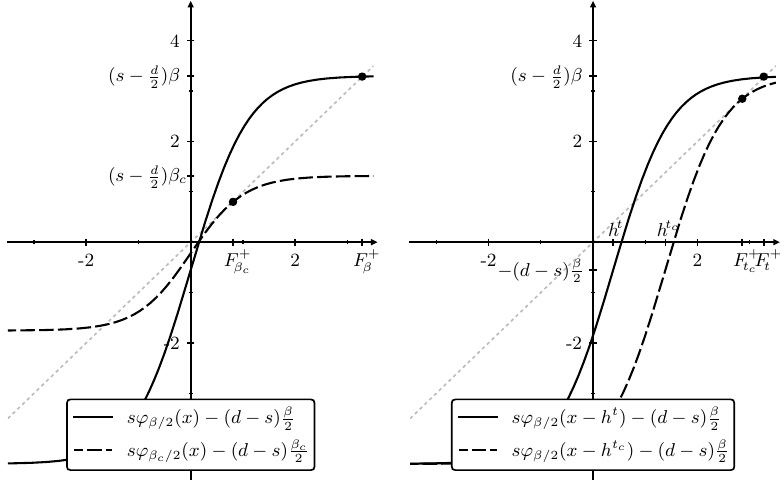}%
\caption{Solutions for the recursion \eqref{eq:recursion_function_lambda} at $\beta=1.1$ $s=7$, $d=8$ and $t=0.5$. The left side depicts the limiting case $h^t = 0$. Increasing the repulsion $\beta$ lowers the interception point of the function with the y-axis while increasing its slope and maximal value. If $s$ fulfils the condition \eqref{eq:subtree_size_condition} for fixed $d$ there exists a critical value $\beta_{c}(d,s)$ such that for all $\beta>\beta_{c}(d,s)$ there are exactly two positive fixed points. In particular this allows a positive fixed point to exist for small $h^t$. As $h^t$ is decreasing in $t$ there exists a lower bound $t_{c}(\beta,d,s)$ for the time such that at least one positive fixed point exists for all $t\geq t_{c}(\beta,d,s)$ and thereby $h^t \leq h^{t_{c}}$.}%
\label{fig:critical_values}%
\end{figure}
It remains to show the existence of the fixed point $F^+>0$. From the analysis of the Ising model on Cayley trees in \cite{georgii_gibbs_2011} we know that the function
\begin{equation}
x \mapsto s\varphi_{\beta/2}(x-h)%
\end{equation}
has a positive fixed point if $\beta>\beta'(s)$ and $h \leq h'(\beta,s)$ where
\begin{equation}
\label{eq:critical_field_value}
h'\left(\beta,s\right)%
= \frac{s-1}{2}\beta + \frac{s-1}{2}\log(s-1) - \frac{s}{2}\log(s)+\mathcal{O}_{\beta}(1).%
\end{equation}
Therefore \eqref{eq:recursion_function_lambda} has a positive solution for parameters $\beta>\beta'(s)$ and $t>0$ that fulfil the condition
\begin{equation}
\label{eq:field_condition}
h^t + (d-s)\frac{\beta}{2} \leq h'(\beta,s).%
\end{equation}
For $s>(d+1)/2$ from \eqref{eq:critical_field_value} it follows that there exists a finite $\beta_{c}(d,s)>\beta'(s)$ such that \eqref{eq:field_condition} is a strict inequality for the limiting case $h^t=0$ for each $\beta > \beta_{c}(d,s)$, as depicted in the left part of Figure \ref{fig:critical_values}. Choosing $t_{c}(\beta,d,s)$ subsequently such that \eqref{eq:field_condition} is an equality for $h^{t_c}$ guarantees that it is fulfilled for all $t \geq t_{c}(\beta,d,s)$, as $h^t$ is a decreasing function in $t$. Therefore a positive fixed point $F^+>0$ exists for all $\beta>\beta_c(d,s)$ and $t\geq t_c(\beta,d,s)$.

An analogous argument shows that in case of the minus configuration on $\Delta\setminus\Lambda$ the boundary fields can be upper bounded by $f_{j0}[\eta_{\Lambda\setminus 0}-_{\Delta\setminus\Lambda}] \leq F^- = -F^+ < 0$ for all $j \in \boundary{0} \cap S$ and all $n \in \mathbb{N}$ for identical critical values. Setting $\epsilon = 2F^+$ concludes the proof.
\end{proof}

\subsection{Subtree percolation}
\label{subsec:subtree_percolation}

In this section we show that a fixed occupied site of the Cayley tree of order $d$ whose spins are distributed by the time-evolved intermediate measure has a positive probability of growing an occupied rooted subtree with $s$ or more children for each $s \leq d-1$ at any repulsion strength $\beta>0$, if the activity is greater than a critical value $\lambda_{b}(d)$. This leads to the typicality of bad configurations in the regime of Theorem \ref{thm:condition_bad_configurations} for large activities.

Note that the spin-flip dynamics \eqref{eq:single_site_kernel} does not change the distribution of occupied sites, therefore it suffices to investigate the intermediate measure without time-evolution. First we present a result regarding the non-extremality of this intermediate measure, which shows the necessity of the general zero-one law \ref{thm:zero-one_law} for possibly non-extremal measures for our results.
\begin{proposition}
\label{lemma:non-extremality_intermediate}
For $d \geq 2$, large repulsion $\beta$ and large activity $\lambda$ the intermediate measure $\mu^{\#}_{\beta,\lambda}$ is non-extremal.
\end{proposition}
\begin{proof}
This result follows from the Kesten-Stigum criterion \cite{kesten_additional_1966} as the in modulus second largest eigenvalue $u_2$ of the transition matrix for the intermediate measure fulfils the condition
\begin{equation}
u_2 > \frac{1}{\sqrt{d}}%
\end{equation}
for large repulsion $\beta$ and large activity $\lambda$ if $d \geq 2$ as we will show.

To calculate the transition matrix, we use Zachary's theorem (Theorem \ref{thm:bl_gibbs_measure}) and by applying the consistency condition for boundary laws \eqref{eq:boundary_law_recursion} we get an expression for the spin distribution of two vertices along an arbitrary edge $\{i,j\} \in E$:
\begin{equation}
\label{eq:edge_distribution}
\mu^{\#}_{\beta,\lambda}(\sigma_i=x, \sigma_j=y)%
= \frac{1}{Z}l_{\beta,\lambda}(x)Q_{\{i,j\}}(x,y)l_{\beta,\lambda}(y)%
\qquad \forall x,y \in \{-1,0,1\}.%
\end{equation}
Here the transfer operator is defined by the Widom-Rowlinson potential through Equation \eqref{eq:transfer_operator}. The boundary law for the intermediate measure is the homogeneous solution of the boundary law recursion \eqref{eq:boundary_law_recursion} with $l_{\beta,\lambda}(-1)=l_{\beta,\lambda}(+1)$ and has the representation
\begin{equation}
\big(l_{\beta,\lambda}(-1),l_{\beta,\lambda}(0),l_{\beta,\lambda}(1)\big)%
=\big(\xi_{\beta,\lambda}\lambda^{-\frac{1}{d+1}},1,\xi_{\beta,\lambda}\lambda^{-\frac{1}{d+1}}\big)%
\end{equation}
where $\xi_{\beta,\lambda}$ is the unique positive solution to the equation
\begin{equation}
\label{eq:xi_equation}
x=\lambda\left(\frac{1+(1+\euler^{-\beta})x}{1+2x}\right)^d,%
\end{equation}
see \cite{kissel_hard-core_2019}. Note, that while the exact value of $\xi_{\beta,\lambda}$ is dependent on $\beta$ it is always contained in the interval $(2^{-d}\lambda,\lambda)$ and can therefore be controlled by the activity $\lambda$. Using the distribution of neighbouring spins \eqref{eq:edge_distribution} with an appropriate normalisation the homogeneous matrix of transition probabilities for the intermediate measure $\smash[b]{\mu^{\#}_{\beta,\lambda}}$ takes the form
\begin{equation}
\label{eq:transition_matrix}
P^{\#}_{\beta,\lambda} = \frac{1}{1+(1+\euler^{-\beta})\xi_{\beta,\lambda}}
\begin{pmatrix}
\xi_{\beta,\lambda} & 1 & \euler^{-\beta}\xi_{\beta,\lambda} \\
\alpha_{\beta,\lambda}\xi_{\beta,\lambda} & \alpha_{\beta,\lambda} & \alpha_{\beta,\lambda}\xi_{\beta,\lambda} \\
\euler^{-\beta}\xi_{\beta,\lambda} & 1 & \xi_{\beta,\lambda} \\
\end{pmatrix}
\end{equation}
where
\begin{equation}
\alpha_{\beta,\lambda} := \frac{1+(1+\euler^{-\beta})\xi_{\beta,\lambda}}{1+2\xi_{\beta,\lambda}}.%
\end{equation}
The eigenvalues of $P^{\#}_{\beta,\lambda}$ ordered by absolute value are:
\begin{equation}
\label{eq:eigenvalues}
u_1 = 1 \qquad%
u_2 = \frac{(1-\euler^{-\beta})\xi_{\beta,\lambda}}{1+(1+\euler^{-\beta})\xi_{\beta,\lambda}} \qquad%
u_3 = -\frac{u_2}{1+2\xi_{\beta,\lambda}}%
\end{equation}
As $\xi_{\beta,\lambda}$ can be solely controlled by $\lambda$ we get
\begin{equation}
\lim_{\lambda \rightarrow \infty} u_2 = \frac{1-\euler^{-\beta}}{1+\euler^{-\beta}} = \tanh(\beta/2)%
\qquad \forall \beta>0.%
\end{equation}
Therefore for any repulsion strength $\beta$ with
\begin{equation}
\label{eq:kesten-stigum_ising}
\tanh(\beta/2) > \frac{1}{\sqrt{d}}%
\end{equation}
and sufficently large $\lambda$ the Kesten-Stigum criterion is fulfilled and the measure is non-extremal.
\end{proof}
The criterion \eqref{eq:kesten-stigum_ising} yields the critical value at which the intermediate measure for the Ising model on the Cayley tree of order $d$ with parameter $\beta/2$ transitions from extremal to non-extremal, for more information on this transition of the intermediate measure for the Ising model see \cite{bleher_purity_1995,gandolfo_glassy_2020,ioffe_extremality_1996,pemantle_critical_2010}. This is consistent with the observation that the Widom-Rowlinson model of repulsion strength $\beta$ conditioned on full occupation yields the Ising model with parameter $\beta/2$.

Next we briefly look at the \emph{occupation measure} describing the distribution of occupied sites. For the intermediate measure $\smash[b]{\mu^{\#}_{\beta,\lambda}}$ this measure proves to be a tree-indexed Markov chain: We define the mapping $\tau:\Omega \rightarrow \{0,1\}^V$ with $\tau(\omega) := (|\omega_i|)_{i\in V}$.
\begin{lemma}
\label{lemma:occupation_markov_chain}
The occupation measure for the intermediate measure defined by $\mu^{\#}_{\beta,\lambda}\circ \tau^{-1}$ on $(\{0,1\}^V,\mathcal{P}(\{0,1\})^{\otimes V})$ is a tree-indexed Markov chain.
\end{lemma}
\begin{proof}
For any $S \subset V$ we define $\mathcal{F}^{oc}_S:=\sigma(|\sigma_i|, i \in S)$ the $\sigma$-algebra generated by the occupation numbers in $S$. Using first the tower property for conditional probabilities and then the Markov-chain property of the intermediate measure we get
\begin{equation}
\label{eq:tower_property}
\begin{split}
\mu^{\#}_{\beta,\lambda}(|\sigma_j| = x | \mathcal{F}^{oc}_{(-\infty,ij)})%
&= \mu^{\#}_{\beta,\lambda}\big(\mu^{\#}_{\beta,\lambda}(|\sigma_j| = x | \mathcal{F}_{(-\infty,ij)})%
\big| \mathcal{F}^{oc}_{(-\infty,ij)}\big) \\%
&= \mu^{\#}_{\beta,\lambda}\big(\mu^{\#}_{\beta,\lambda}(|\sigma_j| = x | \mathcal{F}_{i})%
\big| \mathcal{F}^{oc}_{(-\infty,ij)}\big).%
\end{split}
\end{equation}
Due to the symmetries of the transition matrix \eqref{eq:transition_matrix} of the intermediate measure we have
\begin{equation}
\label{eq:f_measurability}
\mu^{\#}_{\beta,\lambda}(|\sigma_j| = x | \mathcal{F}_{i})(+1)%
= \mu^{\#}_{\beta,\lambda}(|\sigma_j| = x | \mathcal{F}_{i})(-1)%
\qquad \forall x \in \{0,1\}%
\end{equation}
for each edge $\{i,j\}\in E$, which shows that $\mu^{\#}_{\beta,\lambda}(|\sigma_j| = x | \mathcal{F}_{i})$ is $\mathcal{F}^{oc}_i$-measurable. Therefore the last line of \eqref{eq:tower_property} is $\mathcal{F}^{oc}_i$-measurable which yields
\begin{equation}
\mu^{\#}_{\beta,\lambda}(|\sigma_j| = x | \mathcal{F}^{oc}_{(-\infty,ij)})%
= \mu^{\#}_{\beta,\lambda}(|\sigma_j| = x | \mathcal{F}^{oc}_{i}).%
\end{equation}
This holds for arbitrary $x \in \{0,1\}$ and all edges $\{i,j\} \in E$, proving the Markov chain property for the occupation measure.
\end{proof}
From the explicit form of the transition matrix we can estimate the transition probabilities of the occupation measure for large and small activities respectively.
\begin{lemma}
\label{lemma:transition_probability}
The transition probability between neighbouring occupied sites of the $\mu^{\#}_{\beta,\lambda}$-measure can be controlled by the activity $\lambda$, to be more precise
\begin{alignat}{2}
&\lim_{\mathclap{\lambda\rightarrow\infty}}\ \inf_{\beta>0}\ %
&&\mu^{\#}_{\beta,\lambda}(|\sigma_j|=1|\,|\sigma_i|=1) = 1 \\%
&\lim_{\mathclap{\lambda\rightarrow 0}}\ \sup_{\beta>0}\ %
&&\mu^{\#}_{\beta,\lambda}(|\sigma_j|=1|\,|\sigma_i|=1) = 0%
\end{alignat}
for all edges $\{i,j\} \in E$ and for arbitrary fixed $d\in\mathbb{N}$.
\end{lemma}
\begin{proof}
The transition probabilities are given by the transition matrix $P^{\#}_{\beta,\lambda}$ \eqref{eq:transition_matrix} from the proof of Lemma \ref{lemma:non-extremality_intermediate}. Due to the matrix symmetry we get
\begin{equation}
\mu^{\#}_{\beta,\lambda}(|\sigma_j|=1|\,|\sigma_i|=1)%
= \frac{(1+\euler^{-\beta})\xi_{\beta,\lambda}}{1+(1+\euler^{-\beta})\xi_{\beta,\lambda}}.%
\end{equation}
Also from the previous proof we know that $\xi_{\beta,\lambda}$ is bounded by $(2^{-d}\lambda,\lambda)$, therefore we can estimate the transition probability in the following way:
\begin{equation}
\frac{2\lambda}{1+2\lambda}%
> \frac{2\xi}{1+2\xi}%
> \mu^{\#}_{\beta,\lambda}(|\sigma_j|=1|\,|\sigma_i|=1)%
> \frac{\xi}{1+\xi}%
> \frac{2^{-d}\lambda}{1+2^{-d}\lambda}%
\qquad \forall \beta>0.%
\end{equation}
Taking the limit for large and small activities respectively completes the proof.
\end{proof}
Finally, using the estimations for the transition probabilities in Lemma \ref{lemma:transition_probability} we can prove that an active site in the sense of the zero-one law \ref{thm:zero-one_law} creates an occupied rooted tree of sufficient size with positive probability.
\begin{proof}[Proof of Proposition \ref{prop:probability_subtrees}]
Let $p_s(\beta,\lambda,d)$ denote the probability that a fixed occupied site on the Cayley tree of order $d$ whose sites are occupied according to the time-evolved intermediate measure $\mu^{\#}_{\beta,\lambda,t}$ is the root of an occupied subtree where each vertex has at least $s$ children. Note that the distribution of occupied sites does not depend on $t$, as the spin-flip dynamics \eqref{eq:single_site_kernel} does not affect the positions of occupied sites. It is therefore sufficient to work with the measure $\mu^{\#}_{\beta,\lambda}$. We show the existence of a critical activity $\lambda_{b}(d) \in (0,\infty)$ such that
\begin{equation}
p_s(\beta,\lambda,d) > 0%
\end{equation}
for all $\beta > 0$ and all $s \leq d-1$ if $\lambda \geq \lambda_{b}(d)$.

Denote by $p^{(n)}_s(\beta,\lambda,d)$ the probability that an occupied site on the Cayley tree of order $d$ is the root of a finite subtree of $n$ generations where each vertex has at least $s$ occupied children, then we get
\begin{equation}
\lim_{n\rightarrow\infty}p^{(n)}_s(\beta,\lambda,d) = p_s(\beta,\lambda,d).%
\end{equation}
The root site is chosen to be occupied, therefore we have $p^{(0)}_s(\beta,\lambda,d) = 1$ and as the growing of the tree follows a binomial distribution we get the $\beta$-independent inequality
\begin{equation}
p^{(n+1)}_s(\beta,\lambda,d)%
\geq \sum_{k = s}^d%
\binom{d}{k}\big(u_{\lambda}\,p^{(n)}_s(\beta,\lambda,d)\big)^k%
\big(1-u_{\lambda}\,p^{(n)}_s(\beta,\lambda,d)\big)^{d-k} \qquad \forall \beta>0,%
\end{equation}
with $u_{\lambda} := \inf_{\beta>0}\mu^{oc}_{\beta,\lambda}(|\sigma_y|=1|\,|\sigma_x|=1)$. It suffices to show the positivity of $p_{d-1}(\beta,\lambda,d)$, as a subtree of size $d-1$ already contains an $s$ subtree for every $s \leq d-1$.

A fixed point $p_{\lambda} \in (0,1)$ for the mapping
\begin{equation}
\label{eq:tree_recursion}
p \mapsto (u_{\lambda}p)^d + d (u_{\lambda}p)^{d-1}(1-u_{\lambda}p)%
\end{equation}
would be a lower bound for $p^{(n)}_{d-1}(\beta,\lambda,d)$ for all $\beta > 0$ and all $n \in \mathbb{N}$. We show that such a fixed point exists for large activities $\lambda$:

The function
\begin{equation}
g(x) = x^d + d x^{d-1}(1-x)%
\end{equation}
is continuous and takes the values $g(0)=0$, $g(1)=1$ as well as $g'(0)=g'(1)=0$ for each $d \geq 2$. By the intermediate value theorem the set of fixed points of $g$ on $(0,1)$ is non-empty and as $g$ is a polynomial it is finite. Choosing $x_c \in (0,1)$ to be the largest of these fixed points guarantees that $g(x)>x$ for all $x \in (x_c,1)$ as $g'(1)=0$. By Lemma \ref{lemma:transition_probability} we can choose a critical activity $\lambda_{b}(d)$ such that for any $\lambda \geq \lambda_{b}(d)$ the transition probability $u_{\lambda}$ is close enough to $1$ so that there exists a fixed point $x_{\lambda} \in (x_c,u_{\lambda})$ for
\begin{equation}
\label{eq:tree_recursion_reparam}
x \mapsto u_{\lambda}g(x).%
\end{equation}
Substituting $p_{\lambda}:=x_{\lambda}/u_{\lambda} \in (0,1)$ in \eqref{eq:tree_recursion_reparam} shows that $p_{\lambda}$ is a fixed point for \eqref{eq:tree_recursion}. Therefore we have the positive lower bound
\begin{equation}
p_c := \inf_{\lambda \geq \lambda_{b}(d)}p_{\lambda}%
\geq \inf_{\lambda \geq \lambda_{b}(d)} x_{\lambda}%
\geq x_c > 0%
\end{equation}
for $p^{(n)}_{d-1}(\beta,\lambda,d)$ for all $\beta>0$, $\lambda \geq \lambda_{b}(d)$ and all $n \in \mathbb{N}$. Taking the limit gives
\begin{equation}
p_{s}(\beta,\lambda,d) \geq p_{d-1}(\beta,\lambda,d) \geq p_c > 0%
\end{equation}
for all $\beta > 0$ and all $s \leq d-1$ if $\lambda \geq \lambda_{b}(d)$, which proves the proposition.
\end{proof}

\subsection{Proof of Theorem \ref{thm:almost-sure-badness}}
\label{subsec:ultimate_proof}

Combining all previous results we get a proof for the main result:
\begin{proof}
\label{proof:almost-sure-badness}
First, by Theorem \ref{thm:condition_bad_configurations}, for $d \geq 4$ with the choice $s=d-1$ we get finite critical values $\beta_{b}(d)>0$, $t_{b}(\beta,d)>0$ such that all configurations containing an occupied rooted subtree of order $d-1$ are bad for $\beta>\beta_{b}(d)$ and all times $t \geq t_{b}(\beta,d)$. By Proposition \ref{prop:probability_subtrees} there exists a critical activity $\lambda_{b}(d)$ such that these bad configurations have positive probability for all $\lambda \geq \lambda_{b}(d)$. Therefore the set of all bad configurations has positive probability and using the zero-one law \ref{thm:zero-one_law} we see that the set of bad configurations is an almost sure event.
\end{proof}

\section{Proofs: Goodness}
\label{sec:proofs_goodness}

To prove almost-sure Gibbsianness we will use the transition probability
\begin{equation}
u_{\beta,\lambda} := \mu^{\#}_{\beta,\lambda}(|\sigma_y|=1|\,|\sigma_x|=1).%
\end{equation}
and Galton-Watson trees. A Galton-Watson tree is a random rooted tree constructed in the following way. Starting at the root one chooses a random number of children according to a known distribution. Then for every child one independently chooses according to the same distribution again the number of children. This procedure will be repeated for every new child.
\begin{proof}[Proof of Theorem \ref{thm:almost-sure-goodness}]
It is known that a Galton-Watson tree has almost surely no infinite connected component if the expected number of children is smaller than one. In our case the offspring distribution is given by a binomial random variable with probability of success less than $u_{\beta,\lambda}$. If $u_{\beta,\lambda} < 1/d$ there almost surely cannot exist an infinite connected component of occupied sites inside the tree. Since $\lim_{\lambda \rightarrow 0}\sup_{\beta>0} u_{\beta,\lambda} = 0$ by Lemma \ref{lemma:transition_probability} there exists a critical $\lambda_{g}(\beta,d) \in (0,\infty)$ such that $u_{\beta,\lambda} < 1/d$ for all $\lambda < \lambda_{g}(\beta,d)$. Only configurations with an infinite cluster of occupied sites can be bad which implies that for all $\lambda < \lambda_{g}(\beta,d)$ the measure $\mu^\#_{\beta,\lambda,t}$ is almost surely Gibbs.  
\end{proof}
\begin{proof}[Proof of Theorem \ref{thm:almost-sure-goodness_dobrushin}]
In \cite{kissel_dynamical_2020} the authors have investigated the soft-core Widom-Rowlinson models on the lattice $V=\mathbb{Z}^d$. They used Dobrushin-uniqueness theory to show that the time-evolved measure is Gibbs for small times. A quasilocal specification $\gamma$ satisfies the Dobrushin condition if
\begin{equation}
\sup_{i\in V}\;\;\sum_{j\in V} \;\;\sup_{\eta,\zeta\in \Omega :\,\eta_{V\backslash{j} }=\zeta_{V\backslash {j}}}%
\mathrm{d}_{TV,i}(\gamma_{\{i\}}(\,\cdot\,\vert \eta_{V\backslash \{i\}}),%
\gamma_{\{i\}}(\,\cdot\,\vert \zeta_{V\backslash \{i\}})) < 1%
\end{equation}
where $\mathrm{d}_{TV,i}$ is the total variation distance for measures on $(\Omega_{\{i\}},\mathcal{F}_{\{i\}})$. In other words the Dobrushin condition is satisfied if changing only one site in the boundary condition has not a big effect. This condition can be used to prove that there exists a unique Gibbs measure for the specification. Furthermore, one can use the Dobrushin comparison Theorem which is an important ingredient to prove short-time Gibbsianness. The Theorem 2.8 in \cite{kissel_dynamical_2020} states that if the absolute value of the external magnetic field $h$ and $\lambda$ are big enough the soft-core Widom-Rowlinson satisfies the Dobrushin condition. This not only shown for the lattice but also for general locally finite graphs. As Cayley trees are in fact locally finite graphs one can repeat the steps in \cite{kissel_dynamical_2020} to prove that the time-evolved measure is Gibbs. Moreover, this result does not need that the starting measure is  $\mu_{\beta,\lambda}^\#$. It holds for every starting Gibbs measure for the soft-core Widom-Rowlinson model.
\end{proof}
\begin{remark}
In \cite{kissel_dynamical_2020} it is proven that the soft-core Widom-Rowlinson model satisfies the Dobrushin condition if $\beta (d+1)< 2$. Hence for small enough $\beta$ the time-evolved model is Gibbs for all times t>0.
\end{remark}

\printbibliography

\end{document}